\documentclass[a4paper,11pt]{amsart}
\usepackage{amsmath,amsfonts,amsthm,amssymb,enumerate}
\usepackage{graphicx}
\usepackage{comment}
\usepackage{mathtools}
\usepackage{tikz-cd}
\usepackage[setpagesize=false]{hyperref}

\numberwithin{equation}{section}

\newtheorem{theorem}{Theorem}[section]
\newtheorem{corollary}[theorem]{Corollary}
\newtheorem{proposition}[theorem]{Proposition}

\newtheorem{lemma}[theorem]{Lemma}

{
\theoremstyle{definition}
\newtheorem{definition}[theorem]{Definition}
\newtheorem{example}[theorem]{Example}
\newtheorem{remark}[theorem]{Remark}
}

\title{Chirality of torus-covering $T^2$-links of degree three}

\author{Hohto Bekki}
\address{Department of Mathematics, Information Science and  Engineering,  
Saga University,  
1 Honjomachi, Saga, 840-8502, Japan.  
}
\email{bekki@cc.saga-u.ac.jp}

\author{Teruhisa Kadokami}
\address{School of Mechanical Engineering,
College of Science and Engineering,
Kanazawa University,
Kakuma-machi, Kanazawa, Ishikawa, 920-1192, Japan}
\email{kadokami@se.kanazawa-u.ac.jp}

\author{Inasa Nakamura}
\address{Department of Mathematics, Information Science and  Engineering,  
Saga University,  
1 Honjomachi, Saga, 840-8502, Japan.  
}
\email{inasa@cc.saga-u.ac.jp}


\begin{document}  
\begin{abstract}
A torus-covering $T^2$-link of degree $n$ is a surface-link consisting of tori, in the form of an unbranched covering of degree $n$ over the standard torus. We focus on a torus-covering $T^2$-link of degree 3, which is determined by a pair $(a,b)$ of 3-braids satisfying $ab=ba$, denoted by $\mathcal{S}_3(a,b)$. 
We investigate to what extent  the chirality of $\mathcal{S}_3(a,b)$ is detected by invariants such as the triple linking numbers, the number of Fox $p$-colorings, and the quandle cocycle invariant associated with $p$-colorings. In particular, we determine the quandle cocycle invariant for $\mathcal{S}_3(a,b)$ associated with tri-colorings. 
\end{abstract}
\maketitle

\section{Introduction}\label{sec1}

A surface-link is the image of a smooth embedding of a closed surface into the Euclidean 4-space $\mathbb{R}^4$. In this paper, classical links/braids and surface-links are smooth and oriented. 
We treat a certain type of surface-link, called torus-covering $T^2$-links. 
A {\it $T^2$-link} is a surface-link each of whose components is an embedded torus.
A {\it torus-covering $T^2$-link of degree $n$} is a $T^2$-link determined by a pair of commuting $n$-braids $(a,b)$, i.e., satisfying $ab=ba$, called basis $n$-braids, where $n$ is a positive integer.
We denote by $\mathcal{S}_n(a,b)$ the torus-covering $T^2$-link of degree $n$ with basis $n$-braids $(a,b)$. 
The aim of this paper is to investigate the chirality of $\mathcal{S}_n(a,b)$, especially for the case $n=3$. 

Let $F=\mathcal{S}_n(a,b)$. 
First we observe the presentation of the orientation-reversal of $-F$, the mirror image $F^*$, and the orientation-reversed mirror image $-F^*$ of $F$, in terms of basis $n$-braids. Let $\sigma_i$ $(i=1, \ldots, n-1)$ be the $i$-th standard generator of the $n$-braid group. We denote by $e$ and $\tilde{\Delta}$ the trivial $n$-braid and a full twist $(\sigma_1 \sigma_2 \cdots \sigma_{n-1})^n$ of $n$ parallel strings, respectively. 
For an $n$-braid $c$, we denote by $-c$, $c^*$ and $-c^*$ the orientation-reversal, the mirror image, and the orientation-reversed mirror image of $c$, respectively; see Lemma \ref{lemma3-3}.

\begin{theorem}\label{theorem3-2}
Let $(a,b)$ be any $n$-braids which commute. Then we have the following:

\begin{align}
 &-\mathcal{S}_n(a,b) \sim \mathcal{S}_n(-a,b^*)\sim\mathcal{S}_n(a^*, -b), \label{eq3-1}\tag{1}\\
 &\mathcal{S}_n(a,b)^* \sim \mathcal{S}_n(a^*,b^*)\sim\mathcal{S}_n(-a, -b), \label{eq3-2}\tag{2}\\
& -\mathcal{S}_n(a,b)^* \sim \mathcal{S}_n(-a^*,b)\sim\mathcal{S}_n(a, -b^*). \label{eq3-3}\tag{3}
\end{align}
In particular, if $a = -a$, then, for any integer $m$, 
\begin{align}
\mathcal{S}_n(a,\tilde{\Delta}^{m}) \sim \mathcal{S}_n(-a,-\tilde{\Delta}^{m})\sim-\mathcal{S}_n(a, \tilde{\Delta}^{-m}).  \label{eq3-4}\tag{4}
\end{align}
\end{theorem}

We investigate invariants such as the triple linking numbers, the number of $p$-colorings and the quandle cocycle invariant associated with $p$-colorings, where $p$ is an odd prime. 

For a braid $a$, the {\it closure} of $a$, or the {\it closed braid} $\hat{a}$, is the link obtained from $a$ by connecting each $i$-th initial point and $i$-th terminal point by a trivial arc. 
For a surface-link $F$ with equal to or more than three components, 
the triple liking number $\mathrm{Tlk}_{i,j,k}(F)$ $(i \neq j, j \neq k)$ is an invariant of $F$ defined as the total sum of the number of positive triple points of type $(i,j,k)$ minus the number of negative triple points of type $(i,j,k)$; see Section \ref{sec4}. 
We consider three-component $\mathcal{S}_3(a,b)$, which is given by pure 3-braids $a$ and $b$. 
For each $i=1,2,3$, we define the $i$-th component of $\hat{a}$, $\hat{b}$, and $\mathcal{S}_3(a,b)$ to be the component corresponding to the $i$-th strand of $a$ and $b$. 
Then the triple linking numbers are determined from the linking numbers of $\hat{a}$ and $\hat{b}$ (Theorem \ref{thm-tlk}). Using Theorem \ref{thm-tlk}, we have the following corollary. We denote by $\mathrm{Lk}_{i,j}(L)$ the linking number between the $i$-th and $j$-th components of a classical link $L$, and we say that $L$ has {\it non-trivial linking numbers} if $\mathrm{Lk}_{i,j}(L)\neq 0$ for some $i,j$. A surface-link $F$ is said to be {\it reversible} (respectively, {\it $(-)$-amphicheiral}) if $F$ is equivalent to $-F$ (respectively, $-F^*$).

\begin{corollary}\label{cor1-2}
Let $(a, b)$ be pure $3$-braids which commute, such that the closures $\hat{a}$ and $\hat{b}$ have non-trivial linking numbers, and for any given real number $\lambda \neq 0$, $\mathrm{Lk}_{i,j}(\hat{a})\neq \lambda \cdot \mathrm{Lk}_{i,j}(\hat{b})$ for some $i,j \in \{1,2,3\}$, and for any $i,j,k$ with $\{i,j,k\}=\{1,2,3\}$, $\mathrm{Lk}_{i,j}(\hat{a})\neq \mathrm{Lk}_{j,k}(\hat{a})$ or $\mathrm{Lk}_{i,j}(\hat{b})\neq \mathrm{Lk}_{j,k}(\hat{b})$. 
Then $\mathcal{S}_3(a,b)$ is neither reversible nor $(-)$-amphicheiral. 
\end{corollary}

Let $p$ be an odd prime. 
We further study our theme using Fox $p$-colorings. A {\it quandle} is a set with a binary operation satisfying certain conditions, and a {\it $p$-coloring} for a classical link diagram or a surface-link diagram $D$ is a certain map which assign an element of a dihedral quandle $R_p=\mathbb{Z}/p\mathbb{Z}$ to each arc or sheet of $D$. 
We discuss the number of $p$-colorings of the closure of a 3-braid, and we observe the quandle cocycle invariant of $\mathcal{S}_3(a,b)$ associated with $p$-colorings. 
More precisely, we consider the \textit{reduced quandle cocycle invariant} (see Definition \ref{reduced}), which is sufficient to determine the original quandle cocycle invariant, and compute it for torus-covering $T^2$-links of a special form as follows. 
An integer $\nu$ is called a {\it quadratic residue} mod ${p}$ if $\nu\equiv \mu^2 \pmod{p}$ for some $\mu \in \mathbb{Z}/p\mathbb{Z}$, and $\nu$ is called a {\it quadratic non-residue} if it is not a quadratic residue. 
Let $\Big(\frac{~}{p}\Big)$ denote the Legendre symbol, i.e., for $\nu \in \mathbb{Z}$, we have
\begin{align*}
  \Big(\frac{\nu}{p}\Big)
=
\begin{dcases}
  1 &\text{ if $\nu$ is a quadratic residue mod ${p}$ and $\nu \not\equiv 0 \pmod{p}$} \\
  -1 &\text{ if $\nu$ is a quadratic non-residue mod ${p}$ and $\nu \not\equiv 0 \pmod{p}$} \\
  0 &\text{ if $\nu \equiv 0 \pmod{p}$}. 
\end{dcases}
\end{align*}
Furthermore, set 
\begin{align*}
  \varepsilon_p 
=
\begin{dcases}
  1 &\text{ if $p\equiv 1 \pmod{4}$}\\ 
  \sqrt{-1} &\text{ if $p\equiv 3 \pmod{4}$}. 
\end{dcases}
\end{align*}
Then we have the following
 
\begin{theorem}
\label{thm0304}
Let $p\geq 3$ be an odd prime. 
Let $n$ be an odd integer and let $a$ be an $n$-braid presented by 
\[
a=\prod_{j=1}^N \sigma_1^{pk_{1,j}} \sigma_2^{pk_{2,j} } \cdots \sigma_{n-1}^{pk_{n-1,j} }    
\]
for some integer $N>0$ and $k_{1,1}, \ldots, k_{n-1,N} \in \mathbb{Z}$. 
Let $m$ be any integer, and set 
\begin{itemize}
\item 
$\nu_i =\sum_{j=1}^N k_{i,j}$ $(i=1, \ldots, n-1)$, 
\item 
$J = \{ i \in\{1, \dots, n-1\} \mid 2mn \nu_i \not\equiv 0 \pmod{p}\}$. 
\end{itemize}
Then the reduced quandle cocycle invariant $\tilde{\Phi}_p(\mathcal{S}_n(a,\tilde{\Delta}^{2m})) \in \mathbb{C}$ is computed as
\begin{align*}
  \tilde{\Phi}_p(\mathcal{S}_n(a,\tilde{\Delta}^{2m}))
=
p^{n-\frac{1}{2}\#J}\varepsilon_p^{\#J}
\prod_{i \in J}\Big(\frac{2mn\nu_i}{p}\Big). 
\end{align*}
 \end{theorem}

Using Theorem \ref{thm0304}, we obtain the following 
\begin{corollary}
\label{thm1-3}
Let the notation be the same as in Theorem \ref{thm0304}. Then we have 
\begin{align*}
  \tilde{\Phi}_p(\mathcal{S}_n(a,\tilde{\Delta}^{2m})) \neq \overline{\tilde{\Phi}_p(\mathcal{S}_n(a,\tilde{\Delta}^{2m}))}
\end{align*}
if and only if $p\equiv 3 \pmod{4}$ and $\#J$ is odd. 
In particular, $\mathcal{S}_n(a,\tilde{\Delta}^{2m})$ is not $(-)$-amphicheiral if $p\equiv 3 \pmod{4}$ and $\#J$ is odd. 
\end{corollary}

A $p$-coloring for $p=3$ is called a {\it tri-coloring}. 
We investigate tri-colorings, and we classify $\mathcal{S}_3(a,b)$ under an equivalence relation, called the {\it qdl-equivalence} relation (Definirion \ref{def6-6}), which is invariant with respect to the quandle cocycle invariant. 
Though we cannot distinguish the chirality of $\mathcal{S}_3(a,b)$ by quandle cocycle invariant associated with tri-colorings, 
we determine the quandle cocycle invariant as follows. 

\begin{theorem}\label{thm1-4}
For arbitrary 3-braids $(a,b)$ which commute, 
the quandle cocycle invariant $\Phi_3(\mathcal{S}_3(a,b))$ in $\mathbb{Z}[v, v^{-1}]/(v^3-1)$ associated with tri-colorings and the Mochizuki 3-cocycle is the number of tri-colorings of $\mathcal{S}_3(a,b)$, determined as 
\[
\Phi_3(\mathcal{S}_3(a,b))=\begin{cases} 27  & \text{if $\mathcal{S}_3(a,b)$ is qdl-equivalent to $\mathcal{S}_3(e, e)$}\\ 
 9 & \text{if $\mathcal{S}_3(a,b)$ is qdl-equivalent to $\mathcal{S}_3(\sigma_1^{\pm1}, e)$} \\ 
 3 & \text{otherwise.} 
 \end{cases}
 \]
\end{theorem}
 
The paper is organized as follows. In Section \ref{sec2}, we review torus-covering $T^2$-links. In Section \ref{sec2}, we discuss equivalence of $\mathcal{S}_n(a,b)$ and we show Theorem \ref{theorem3-2}. In Section \ref{sec3}, we review $p$-colorings. In Section \ref{sec4}, we discuss the triple linking numbers, and we show Corollary \ref{cor1-2}. In Section \ref{sec5}, we observe the number of $p$-colorings. In Section \ref{sec6-1}, we review the quandle cocycle invariant associated with $p$-colorings. 
In Section \ref{0306-new-sec6-2}, we define the reduced quandle cocycle invariant and prove Theorem \ref{thm0304} and Corollary \ref{thm1-3}. 
In Section \ref{sec6-2}, 
we focus on tri-colorings and classify $\mathcal{S}_3(a,b)$ under the qdl-equivalence relation, which is invariant under the quandle cocycle invariant, 
and then we prove Theorem \ref{thm1-4}. 
In Section \ref{sec8}, we discuss other results derived from Theorem \ref{thm1-4}. 
We refer to \cite{CKS, CF, Kawauchi} for basics of classical knot and surface-knot theory. 

\section{Torus-covering $T^2$-links}\label{sec2}
In this section, we review torus-covering $T^2$-links \cite{N}. 
 A {\it surface-link} is an oriented closed surface smoothly embedded in $\mathbb{R}^4$, and two surface-links are said to be {\it equivalent} if one is carried to the other by an orientation-preserving self-homeomorphism of $\mathbb{R}^4$. 

Let $T$ be a torus standardly embedded in $\mathbb{R}^4$, i.e., $T$ is the boundary of an unknotted solid torus in $\mathbb{R}^3 \times \{0\} \subset \mathbb{R}^4$. Let $N(T)$ be a tubular neighborhood of $T$ in $\mathbb{R}^4$. Let $n$ be a positive integer. 
\begin{definition}
A surface-link $F$ in $\mathbb{R}^4$ is called a {\it torus-covering $T^2$-link of degree $n$} if it is contained in $N(T) \subset \mathbb{R}^4$ and 
$\mathbf{p} |_F: F \to T$ is an orientation-preserving unbranched covering map of degree $n$, where $\mathbf{p}: N(T) \to T$ is the natural projection. 
\end{definition}

Let $F$ be a torus-covering $T^2$-link. We identify $T=S^1 \times S^1$ with $S^1=[0,1]/(0 \sim 1)$ and $N(T)=D^2 \times T$. Let $\mathbf{m}=S^1 \times \{0\}$ and $\mathbf{l}= \{0\} \times S^1$, a meridian and a longitude of $T$ with the base point $x_0=(0,0)$.  The condition that $F$ is an unbranched covering over $T$ implies that the intersections $F \cap \mathbf{p}^{-1}(\mathbf{m})$ and $F \cap \mathbf{p}^{-1}(\mathbf{l})$ are closures of classical $n$-braids in solid tori $\mathbf{p}^{-1}(\mathbf{m})=D^2 \times S^1 \times \{0\}$ and $\mathbf{p}^{-1}(\mathbf{l})=D^2 \times \{0\} \times S^1$, respectively. 
Taking the starting/terminal point set of the $n$-braids in the 2-disk $\mathbf{p}^{-1}(x_0)=D^2 \times \{(0,0)\}$, we have a pair of $n$-braids, called {\it basis $n$-braids}.

For $n$-braids $a$ and $b$, we say that $a$ and $b$ {\it commute} if $ab=ba$ as elements of the $n$-braid group. 
For a torus-covering $T^2$-link, basis $n$-braids commute, and for any 
pair of $n$-braids $(a,b)$ which commute, there exists a unique torus-covering $T^2$-link of degree $n$ with basis $n$-braids $(a,b)$. For $n$-braids $(a,b)$ which commute, we denote by $\mathcal{S}_n(a,b)$ the torus-covering $T^2$-link with basis $n$-braids $(a,b)$.

\section{Chirality derived from the structure of $\mathcal{S}_n(a,b)$}\label{sec3}
In this section, we discuss the equivalence of $\mathcal{S}_n(a,b)$ and prove Theorem \ref{theorem3-2}. 
We assume that $N(T)=D^2 \times T$ is embedded in $\mathbb{R}^4$ as follows: 
\begin{itemize}
\item
$S^1=\{(u,v) \in \mathbb{R}^2 \mid u^2+v^2=1\}$, 

\vspace{0.2cm}
\item
$T=\left\{\bigl(q_1(2+p_1), q_2(2+p_1), p_2, 0\bigl) \mid (p_1, p_2), (q_1, q_2) \in S^1\right\}$,
\vspace{0.2cm}
\item
$\mathbf{m}=\{\left(2+p_1, 0, p_2, 0\right) \mid (p_1, p_2)\in S^1 \}$,
\vspace{0.2cm}
\item
$\mathbf{l} =\{\left(3q_1, 3q_2, 0, 0\right) \mid (q_1, q_2) \in S^1\}$,
\vspace{0.2cm}
\item
$x_0=(3,0,0,0)$,
\vspace{0.2cm}
\item
For $x =\bigl((\cos \phi) (2+\cos \theta), (\sin \phi) (2+\cos \theta), \sin \theta, 0\bigl)  \in T$, 
\begin{align*}
\quad \quad \quad &D^2 \times \{ x\}  &\\
\quad \quad \quad & =
\left\{\bigl((\cos \phi) (2+r\cos \theta), (\sin \phi) (2+r\cos \theta), r\sin \theta, t\bigl) \mid 1/4 \leq r \leq 5/4, -1 \leq t \leq 1 \right\} &\\
\quad \quad \quad &=\mathbf{p}^{-1}( x ). &
\end{align*}
\end{itemize}

We equip $D^2$ or $T$ with positive orientation, and we denote by $-D^2$ or $-T$  the manifolds obtained from $D^2$ or $T$ by reversing the orientation. 
We denote by $(-\mathbf{m}, -\mathbf{l})$ the orientation-reversal of $(\mathbf{m}, \mathbf{l})$. 

\subsection{Equvalence of $\mathcal{S}_n(a,b)$}
\begin{theorem}\label{thm3-1}
Let $(a,b)$ and $(a', b')$  be pairs of $n$-braids which commute. 
We have 
\begin{equation}
\mathcal{S}_n(a,b) \sim \mathcal{S}_n(a^{-1}, b^{-1}). \label{E0} \tag{E1}\\
\end{equation}
If $(a,b)$ and $(a', b')$ are conjugate, 
then $\mathcal{S}_n(a,b)$ and $\mathcal{S}_n(a',b')$ are equivalent, i.e., 

\begin{equation}
\mathcal{S}_n(c^{-1}ac, c^{-1}bc) \sim \mathcal{S}_n(a,b) \label{eq1121-1}\tag{E2}
\end{equation}
for any $n$-braid $c$. 

Further, we have the following relations. 
\begin{align}
& \mathcal{S}_n(a,b) \sim \mathcal{S}_n(b^{-1}, a), \label{eq1121-2} \tag{E3}\\
& \mathcal{S}_n(a,b) \sim \mathcal{S}_n(a, a^2b). \label{eq1121-3} \tag{E4}
\end{align}
\end{theorem}

\begin{proof}
The relations (\ref{eq1121-2}) and (\ref{eq1121-3}) are shown in \cite[Corollary 2.9]{N}.

We show (\ref{E0}). Put $F=\mathcal{S}_n(a,b)$. Let $f$ be an orientation-preserving self-homeomorphism of $\mathbb{R}^4$ given by $f(x,y,z,t)=(x, -y, -z, t)$. Note that $f^2=\mathrm{id}$. Then $f(T)=T$ as oriented manifolds and $(f(\mathbf{m}), f(\mathbf{l}))=(-\mathbf{m}, -\mathbf{l})$. 
Since $\mathbf{p}|_F: F \to T$ is an orientation-preserving unbranched covering map, so is $f \circ \mathbf{p} \circ f^{-1}|_{f(F)}: f(F) \to f(T)=T$. Note that the restriction of $f$ to $\mathbf{p}^{-1}(x)=D^2 \times \{x\}$ for any $x \in T$ is an orientation-preserving homeomorphism $D^2 \times \{x\} \to D^2 \times \{f(x)\}$. 
We see that $f(F) \cap \mathbf{p}^{-1}(\mathbf{m})=f(F) \cap (D^2 \times \mathbf{m})=f(F) \cap f(D^2 \times (-\mathbf{m}))=f(F \cap (D^2 \times (-\mathbf{m})))$. 
Let $a=\sigma_{i_1}^{\epsilon_1} \sigma_{i_2}^{\epsilon_2} \cdots \sigma_{i_k}^{\epsilon_k}$  $(i_1, \ldots, i_k \in \{1, \ldots, n-1\}$, $\epsilon_1, \ldots, \epsilon_k \in \{+1, -1\})$. 
The order of crossings of the closed braid $\hat{a'} \coloneqq f(F) \cap (D^2 \times \mathbf{m})$ is reversed from that of $\hat{a}$, and 
the sign of each crossing of $a'$ 
is changed from that of the corresponding crossing of $a$; so the braid $a'$ is $\sigma_{i_k}^{-\epsilon_k} \sigma_{i_{k-1}}^{-\epsilon_{k-1}} \cdots \sigma_{i_{1}}^{-\epsilon_1}$, which is $a^{-1}$. 
Similarly, 
we see that $f(F) \cap \mathbf{p}^{-1}(\mathbf{l})$ is  the closure of $b^{-1}$. Thus, $f(F)$, which is equivalent to $F$, is $\mathcal{S}_n(a^{-1}, b^{-1})$.

We show (\ref{eq1121-1}). 
Put $F=\mathcal{S}_n(c^{-1}ac, c^{-1}bc)$. 
We consider $N(T)=D^2 \times S^1 \times S^1$. 
Let $t_0=(\cos \theta_0, \sin \theta_0)$ be a point in $S^1$ 
such that the closure of $c^{-1}bc$ in $D^2 \times  \{(\cos 0, \sin 0)\} \times \mathbf{l}$ with the starting point set in $D^2 \times \{x_0\}=D^2 \times \{(\cos 0, \sin 0)\} \times \{(\cos 0, \sin 0)\}$ is interpreted as the closure of $bcc^{-1}$ with the starting point set in $D^2 \times \{(\cos 0, \sin 0)\} \times \{t_0\}$. 
Then, taking $((\cos 0, \sin 0), t_0)$ as a new base point of $S^1 \times S^1=T$, we see that $\mathcal{S}_n(c^{-1}ac, c^{-1}bc) \sim \mathcal{S}_n(a,bcc^{-1})$. More precisely, let $f$ be a  linear transformation of $\mathbb{R}^4$ given by $\begin{pmatrix} \cos \theta_0 & -\sin \theta_0 & 0 & 0 \\
\sin \theta_0 & \cos \theta_0 & 0 & 0 \\
0 & 0 & 1 & 0 \\
0 & 0 & 0 & 1
\end{pmatrix}$, which is an orientation-preserving self-homeomorphism of $\mathbb{R}^4$. Then $f(F)=\mathcal{S}_n(a, bcc^{-1})= \mathcal{S}_n(a, b)$. 
\end{proof}

\subsection{Proof of Theorem \ref{theorem3-2}}

\begin{lemma}\label{lemma3-3}
Let $a$ be an $n$-braid with the presentation $a=\sigma_{i_1}^{\epsilon_1} \sigma_{i_2}^{\epsilon_2} \cdots \sigma_{i_k}^{\epsilon_k}$  $(i_1, \ldots, i_k \in \{1, \ldots, n-1\}$, $\epsilon_1, \ldots, \epsilon_k \in \{+1, -1\})$. Then 
\begin{eqnarray*}
-a &=& \sigma_{i_k}^{\epsilon_k} \sigma_{i_{k-1}}^{\epsilon_{k-1}} \cdots \sigma_{i_1}^{\epsilon_1}, \\
a^* &=& \sigma_{i_1}^{-\epsilon_1} \sigma_{i_2}^{-\epsilon_2} \cdots \sigma_{i_k}^{-\epsilon_k}, \\
-a^* &=& \sigma_{i_k}^{-\epsilon_k} \sigma_{i_{k-1}}^{-\epsilon_{k-1}} \cdots \sigma_{i_1}^{-\epsilon_1}=a^{-1}. 
\end{eqnarray*}
\end{lemma}
\begin{proof}
By definition of the orientation reversed image $-a$ and the mirror image $a^*$ of an $n$-braid $a$, we have the requires result. 
\end{proof}

\begin{proof}[Proof of Theorem \ref{theorem3-2}]
By (\ref{E0}) in Theorem \ref{thm3-1} and $c^{-1}=-c^*$ $(c=a,b)$ (Lemma \ref{lemma3-3}), it suffices to show the first equivalence for each case. 
Put $F=\mathcal{S}_n(a,b)$, which is an orientation-preserving unbranched covering over $T$ with the meridian $\mathbf{m}$ and the longitude $\mathbf{l}$. %

We show (\ref{eq3-1}).  Note that $-F$ is a surface-link in $(-D^2) \times (-T)$ which is in the form of an orientation-preserving unbranched covering over $-T$, with the projection $\mathbf{p}': (-D^2) \times (-T) \to -T$, and $-F$ coincides with $F$ when we forget the orientation. 
We take an orientation-preserving self-homeomorphism $f$ of $\mathbb{R}^4$ satisfying $f((-D^2) \times \{x\})=D^2 \times \{f(x)\}$ $(x \in T)$ and $f(-T)=T$ as the homeomorphism given by $f(x,y,z,t)=(x, y, -z, -t)$. Note that $f^2=\mathrm{id}$. 
Then $f(\mathbf{m})=-\mathbf{m}$ and $f(\mathbf{l})=\mathbf{l}$, and the restriction of $f$ to $(\mathbf{p}')^{-1}(x)=(-D^2) \times \{x\}$ for any $x \in T$ is an orientation-preserving homeomorphism $(-D^2) \times \{x\} \to D^2 \times \{f(x)\}$. 
We see that $f(-F) \cap (D^2 \times \mathbf{m})=f(-F) \cap f((-D^2) \times (-\mathbf{m}))=f((-F) \cap ((-D^2) \times (-\mathbf{m})))$. Let $a=\sigma_{i_1}^{\epsilon_1} \sigma_{i_2}^{\epsilon_2} \cdots \sigma_{i_k}^{\epsilon_k}$  $(i_1, \ldots, i_k \in \{1, \ldots, n-1\}$, $\epsilon_1, \ldots, \epsilon_k \in \{+1, -1\})$. 
The order of crossings of the closed braid $\hat{a'} \coloneqq f(-F) \cap (D^2 \times \mathbf{m})$ is reversed from that of $\hat{a}$, and 
the sign of each crossing of $a'$ 
is unchanged from that of the corresponding crossing of $a$; so the braid $a'$ is $\sigma_{i_k}^{\epsilon_k} \sigma_{i_{k-1}}^{\epsilon_{k-1}} \cdots \sigma_{i_{1}}^{\epsilon_1}$, which is $-a$ by Lemma \ref{lemma3-3}. 
Similarly, $f(-F) \cap (D^2 \times \mathbf{l})=f((-F) \cap f((-D^2) \times \mathbf{l})=f((-F) \cap ((-D^2) \times \mathbf{l})$. 
Let $b=\sigma_{j_1}^{\delta_1} \sigma_{j_2}^{\delta_2} \cdots \sigma_{j_l}^{\delta_l}$  $(j_1, \ldots, j_l \in \{1, \ldots, n-1\}$, $\delta_1, \ldots, \delta_l \in \{+1, -1\})$. 
Then $f(-F) \cap (D^2 \times \mathbf{l})$ is the closure of an $n$-braid $b'$ 
presented by $b'=\sigma_{j_1}^{-\delta_1} \sigma_{j_2}^{-\delta_2} \cdots \sigma_{j_l}^{\delta_l}$, which is $b^*$ by Lemma \ref{lemma3-3}. 

We show (\ref{eq3-2}). 
The mirror image $F^*$ is the image of $F$ by an orientation-reversing self-homeomorphism $f$ of $\mathbb{R}^4$. We take $f$ which maps $(x,y,z,t)$ to $(x,y,z, -t)$. Then $F^*=f(F) \subset (-D^2) \times T$, which is in the form of an orientation-preserving unbranched covering of $T$. 
Then we see that the signs of crossings of $F^* \cap \mathbf{p}^{-1}(\mathbf{m})$ and $F^* \cap \mathbf{p}^{-1}(\mathbf{l})$ are reversed, and they are the closures of the mirror images $a^*$ and $b^*$, respectively. 

The equivalence  
 (\ref{eq3-3}) is obtained from the combination of (\ref{eq3-1}) and (\ref{eq3-2}). 
 The equivalence  (\ref{eq3-4}) follows from the fact that $\tilde{\Delta}=-\tilde{\Delta}$ and $-b^*=b^{-1}$ for any $n$-braid $b$, and  (\ref{eq3-1}).  
\end{proof}

\begin{remark}
Any torus-covering $T^2$-link $\mathcal{S}_n(a,b)$ is presented by a certain finite oriented graph on $T$ called a \lq\lq chart'' \cite{N}. Let $\Gamma$ be a chart on $T$ presenting $\mathcal{S}_n(a,b)$. 
We denote by $-\Gamma$ the chart obtained from $\Gamma$ by reversing the orientation of every edge in $\Gamma$, and we denote by $\Gamma^*$ the mirror image of $\Gamma \subset T \subset \mathbb{R}^3 \times \{0\}$, given by $\Gamma^*=f(\Gamma)$ where $f$ is an orientation-reversing self-homeomorphism of $\mathbb{R}^3 \times \{0\}$. 
Then, $-\mathcal{S}_n(a,b)$ is presented by the chart $\Gamma^*$, and $\mathcal{S}_n(a,b)^*$ is presented by the chart $-\Gamma$. 
\end{remark}

\section{Invariants of torus-covering $T^2$-links of degree 3. \newline (I) Triple linking numbers}\label{sec4}

\subsection{Triple linking numbers}\label{sec4-1}

For a link $L$ or a surface-link $F$, we obtain a diagram of $L$ or $F$ by a method as follows. 
We take the image of $L$ or $F$ by a generic projection to $\mathbb{R}^2$ or $\mathbb{R}^3$. Around a crossing or a double point curve, the image consists of an {\it over-arc/sheet} and an {\it under-arc/sheet} with respect to the projection. 
In order to equip the image with crossing information, we break each under-arc or under-sheet into two pieces around each crossing or double point curve. A {\it diagram} of $L$ (respectively, $F$) is the set consisting of resultant arcs (respectively, compact surfaces), which are also called {\it over-arcs/under-arcs}, or simply {\it arcs} (respectively,  {\it over-sheets/under-sheets}, or simply {\it sheets}). Around a triple point, a diagram consists of a single {\it top sheet}, two {\it middle sheets}, and four {\it bottom sheets}. 
A crossing is called a {\it positive} (respectively, {\it negative}) crossing if the pair of normal vectors $(\mathbf{v}_o, \mathbf{v}_u)$ of the over-arc and under-arcs coincides with the right-handed orientation of $\mathbb{R}^2$. Similarly, a triple point is called a {\it positive} (respectively, {\it negative}) triple point if the triple of normal vectors $(\mathbf{v}_t, \mathbf{v}_m, \mathbf{v}_b)$ of the top, middle, and bottom sheets coincides with the right-handed orientation of $\mathbb{R}^3$. 

For a link $L$ with at least two components, the {\it linking number} of $L$, denoted by $\mathrm{Lk}_{i,j}(L)$ for positive integers $i,j$ with $i \neq j$, is given by 

\[
\mathrm{Lk}_{i,j}(L)=\sum_{\tau \in X_2(i,j)} \epsilon(\tau), 
\]
where $X_2(i,j)$ is the set of crossings of a diagram of $L$ such that the over-arc (respectively, under-arcs) is from the $i$-th (respectively, $j$-th) component, and $\epsilon(\tau)=+1$ (respectively, $-1$) if $\tau$ is a positive (respectively, negative) crossing. 

For a surface-link $F$ with three components (or at least three components), the triple linking numbers $\mathrm{Tlk}_{i,j,k}(F)$ are defined as follows. 
Let $i,j, k$ be positive integers with $i \neq j$ and $j \neq k$. Let $X_3(i,j,k)$ be the set of triple points of a diagram of $F$ such that the top, middle, and bottom sheets are from the $i$-th, $j$-th and $k$-th components, respectively, called  triple points {\it of type $(i,j,k)$}, and for each triple point $\tau$, put $\epsilon(\tau)=+1$ (respectively, $-1$) if $\tau$ is a positive (respectively, negative) triple point; see Figure \ref{fig20251220}. Then, the \textit{triple linking number} between the $i$-th, $j$-th, and $k$-th components of $F$, denoted by $\mathrm{Tlk}_{i,j,k}(F)$, is given by 
\[
\mathrm{Tlk}_{i,j,k}(F)=\sum_{\tau \in X_3(i,j,k)} \epsilon(\tau). 
\]
It is known \cite{CJKLS} that $\mathrm{Tlk}_{k,j,i}(F)=-\mathrm{Tlk}_{i,j,k}(F)$ if $i,j,k$ are mutually distinct, and $\mathrm{Tlk}_{i,j,k}(F)=0$ otherwise. 

\begin{figure}[ht]
\includegraphics[height=4cm]{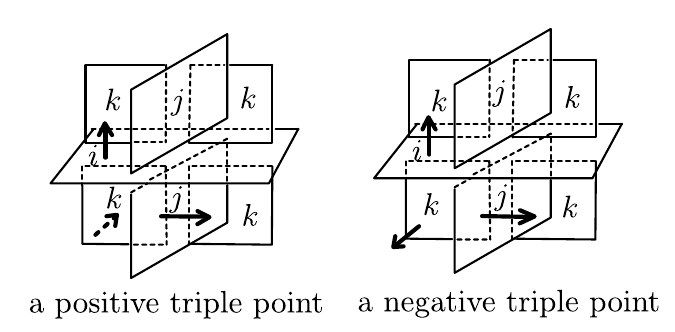}
\caption{A positive triple point and a negative triple point of type $(i,j,k)$.}\label{fig20251220}
\end{figure}

When basis $n$-braids $a$ and $b$ are pure $n$-braids, 
we define, for each $i=1,\ldots, n$, 
the $i$-th component of $\hat{a}$, $\hat{b}$, and $\mathcal{S}_n(a,b)$ to be the component corresponding to the $i$-th strand of $a$ and $b$. 
\begin{theorem}[{\cite[Theorem 1.1]{N2}}]\label{thm-tlk}
Let $(a,b)$ be pure $n$-braids which commute $(n \geq 3)$. 
Then the triple linking numbers of $F=\mathcal{S}_n(a,b)$ are computed as 
\begin{equation*}
\mathrm{Tlk}_{i,j,k}(F)=-\mathrm{Lk}_{i,j}(\hat{a})\mathrm{Lk}_{j,k}(\hat{b})+\mathrm{Lk}_{i,j}(\hat{b})\mathrm{Lk}_{j,k}(\hat{a}). 
\end{equation*}
In particular, 
\begin{equation*}
\begin{pmatrix}
\mathrm{Tlk}_{1,2,3}(F)\\
\mathrm{Tlk}_{2,3,1}(F)\\
\mathrm{Tlk}_{3,1,2}(F)\\
\end{pmatrix}
=
-\begin{pmatrix}
 \mathrm{Lk}_{3,1}(\hat{a}) \\
\mathrm{Lk}_{1,2}(\hat{a})\\
\mathrm{Lk}_{2,3}(\hat{a})
\end{pmatrix}
\times 
\begin{pmatrix}
 \mathrm{Lk}_{3,1}(\hat{b}) \\
\mathrm{Lk}_{1,2}(\hat{b})\\
\mathrm{Lk}_{2,3}(\hat{b})
\end{pmatrix}, 
\end{equation*}
where $\times$ denotes the outer product. 
\end{theorem}

\subsection{Triple linking numbers of $\mathcal{S}_3(a,b)$}
We consider the case of degree 3. 
Put $\Sigma_1=\sigma_1^2$ and $\Sigma_2=\sigma_2^2$. Let $U$ be the free group with two generators $\Sigma_1$ and $\Sigma_2$, and let $Z$ be an infinite cyclic group generated by $\tilde{\Delta}$.  In the proof of \cite[Theorem 1.1]{N2}, we showed that 
the pure 3-braid group $P_3$ is decomposed as the internal direct product $U \times Z$ via $(d, h) \mapsto dh$ ($d \in U$, $h \in Z$): 
\[
P_3=UZ
\cong U \times Z. 
\]

Hence, any pure 3-braids $(a,b)$ which commute are presented by  
\[
a=d^{l_1} \tilde{\Delta}^{m_1},\ b=d^{l_2} \tilde{\Delta}^{m_2},
\]
where $d \in U$ and $l_1, l_2, m_1, m_2$ are integers. 
For the presentation of a $3$-braid $c$ in $P_3=UZ$ and $\Sigma \in \{\Sigma_1, \Sigma_2, \tilde{\Delta}\}$, the {\it algebraic sum} of the numbers of $\Sigma$'s in $c$ is the total sum of the number of $\Sigma$ minus the number of $\Sigma^{-1}$.  

\begin{proposition}\label{prop4-2a}
For a pure 3-braid $c$, 
let $\alpha$, $\beta$, $\delta$ be the algebraic sums of the numbers of $\Sigma_1$'s, $\Sigma_2$'s and $\tilde{\Delta}$'s in the  presentation of $c$ in $P_3=UZ$, respectively. 
Then we have 
\begin{equation*}
\begin{pmatrix}
\mathrm{Lk}_{3,1}(\hat{c})\\
\mathrm{Lk}_{1,2}(\hat{c})\\
\mathrm{Lk}_{2,3}(\hat{c})
\end{pmatrix}
=M \begin{pmatrix}
\alpha\\
\beta\\
\delta
\end{pmatrix}, 
\end{equation*}
where 
$
M=\begin{pmatrix}
0 & 0 & 1 \\
1& 0 & 1\\
0 & 1 & 1
\end{pmatrix}. 
$
\end{proposition}

\begin{proof}
The linking number of $\hat{c}$ is obtained by 
\[
\mathrm{Lk}_{i,j}(\hat{c})=\alpha \cdot \mathrm{Lk}_{i,j}(\hat{\Sigma_1})+\beta \cdot  \mathrm{Lk}_{i,j}(\hat{\Sigma_2}) +\delta \cdot \mathrm{Lk}_{i,j}(\hat{\tilde{\Delta}}). 
\]
Since 
\begin{eqnarray*}
\mathrm{Lk}_{3,1}(\hat{\Sigma_1})=0, & \mathrm{Lk}_{3,1}(\hat{\Sigma_2})=0, & \mathrm{Lk}_{3,1}(\hat{\tilde{\Delta}})=1, \\
\mathrm{Lk}_{1,2}(\hat{\Sigma_1})=1, & \mathrm{Lk}_{1,2}(\hat{\Sigma_2})=0, & \mathrm{Lk}_{1,2}(\hat{\tilde{\Delta}})=1, \\
\mathrm{Lk}_{2,3}(\hat{\Sigma_1})=0, & \mathrm{Lk}_{2,3}(\hat{\Sigma_2})=1, & \mathrm{Lk}_{2,3}(\hat{\tilde{\Delta}})=1, 
\end{eqnarray*}
we have the required result. 
\end{proof}

Theorem \ref{thm-tlk} and Proposition \ref{prop4-2a} imply the following

\begin{corollary}
Let $(a,b)$ be pure 3-braids which commute. 
Put $F=\mathcal{S}_3(a,b)$.  
For $c=a,b$, let $\alpha_c$, $\beta_c$, $\delta_c$ be the algebraic sums 
of the numbers of $\Sigma_1$'s, $\Sigma_2$'s and $\tilde{\Delta}$'s in the presentation of $c$ in $P_3=UZ$, respectively. 
Then the triple linking numbers are computed as 
\begin{equation*}
\begin{pmatrix}
\mathrm{Tlk}_{1,2,3}(F)\\
\mathrm{Tlk}_{2,3,1}(F)\\
\mathrm{Tlk}_{3,1,2}(F)\\
\end{pmatrix}
=
-M \begin{pmatrix}
\alpha_a \\
\beta_a\\
\delta_a
\end{pmatrix}
\times M \begin{pmatrix}
\alpha_b \\
\beta_b\\
\delta_b
\end{pmatrix}, 
\end{equation*}
where $M$ is the matrix given in Proposition \ref{prop4-2a} and $\times$ denotes the outer product. 
\end{corollary}

\begin{example}\label{eg4-4}
When $a=\Sigma_1^{n_1} \Sigma_2^{n_2}$ and $b=\tilde{\Delta}^m$ for integers $n_1, n_2$ and $m$, the triple linking numbers of $F=\mathcal{S}_3(a,b)$ are computed as follows: 
\begin{equation*}
\begin{pmatrix}
\mathrm{Tlk}_{1,2,3}(F)\\
\mathrm{Tlk}_{2,3,1}(F)\\
\mathrm{Tlk}_{3,1,2}(F)\\
\end{pmatrix}
=-m \cdot \begin{pmatrix}
0\\
n_1 \\
n_2 
\end{pmatrix}
\times 
\begin{pmatrix}
1 \\
1 \\
1
\end{pmatrix}.
\end{equation*}
\end{example}

\subsection{Proof of Corollary \ref{cor1-2}}

We give relations of the triple linking numbers between $\mathcal{S}_3(a,b)$ and its orientation-reversed/mirror image. 
\begin{theorem}\label{thm7-1}
Let $(a, b)$ be pure $3$-braids which commute. Then 
the triple linking numbers of $F=\mathcal{S}_3(a,b)$ satisfy the following. 
\begin{eqnarray*}
&&\mathrm{Tlk}_{i,j,k}(-F)=-\mathrm{Tlk}_{i,j,k}(F), \label{eq7-1}\\
&&\mathrm{Tlk}_{i,j,k}(F^*)=\mathrm{Tlk}_{i,j,k}(F), \\
&&\mathrm{Tlk}_{i,j,k}(-F^*)=-\mathrm{Tlk}_{i,j,k}(F). 
\end{eqnarray*}
\end{theorem}

\begin{proof}
By Lemma \ref{lemma3-3}, 
\begin{align*}
&\mathrm{Lk}_{i,j}(-\hat{c})=\mathrm{Lk}_{i,j}(\hat{c}), \\
&\mathrm{Lk}_{i,j}(\hat{c^*})=\mathrm{Lk}_{i,j}(-\hat{c^*})=-\mathrm{Lk}_{i,j}(\hat{c})
\end{align*} for $c=a,b$. 
Thus, Theorems \ref{theorem3-2} and \ref{thm-tlk}  imply the required result. 
\end{proof}

\begin{theorem}\label{thm7-2}
Let $(a, b)$ be pure $3$-braids which commute. 
Let $F=\mathcal{S}_3(a,b)$. Let $F'$ be a surface-link obtained from $F$ by changing the numbering of the components. Assume that $\hat{a}$ and $\hat{b}$ have non-trivial linking numbers, and for any given real number $\lambda \neq 0$, $\mathrm{Lk}_{i,j}(\hat{a})\neq \lambda \cdot \mathrm{Lk}_{i,j}(\hat{b})$ for some $i,j \in \{1,2,3\}$. 
If $\mathrm{Tlk}_{i,j,k}(-F')=\mathrm{Tlk}_{i,j,k}(F)$ for any $i,j,k$ or $\mathrm{Tlk}_{i,j,k}(-(F')^*)=\mathrm{Tlk}_{i,j,k}(F)$ for any $i,j,k$, then $\mathrm{Lk}_{s,t}(\hat{a})=\mathrm{Lk}_{t,u}(\hat{a})$ and  $\mathrm{Lk}_{s,t}(\hat{b})=\mathrm{Lk}_{t,u}(\hat{b})$ for some $s,t,u$ with $\{s,t,u\}=\{1,2,3\}$.  

\end{theorem}
\begin{proof}
By Theorem \ref{thm7-1}, if $\mathrm{Tlk}_{i,j,k}(-F')=\mathrm{Tlk}_{i,j,k}(F)$ for any $i,j,k$ or $\mathrm{Tlk}_{i,j,k}(-(F')^*)=\mathrm{Tlk}_{i,j,k}(F)$ for any $i,j,k$, then the multi-set 
\[
Tlk(F) \coloneqq \{\mathrm{Tlk}_{1,2,3}(F), \mathrm{Tlk}_{2,3,1}(F), \mathrm{Tlk}_{3,1,2}(F)\}
\]
 satisfies $Tlk(F)=-Tlk(F)$, where $-Tlk(F)$ is the multi-set obtained from taking $-\mu$ for any $\mu \in Tlk(F)$; thus the vector 
 \[
 \vec{Tlk}(F) \coloneqq ( \mathrm{Tlk}_{1,2,3}(F), \mathrm{Tlk}_{2,3,1}(F), \mathrm{Tlk}_{3,1,2}(F))^T
 \]
  is in the subspace $W \coloneqq  \{(x_1, x_2, x_3)^T \mid x_s=-x_t,\, x_u=0\}$ for some $s,t,u \in \{1,2,3\}$ with $\{s,t,u\}=\{1,2,3\}$ in the vector space $\mathbb{R}^3$. 
Since the triple linking numbers are presented as the outer product of the linking numbers of $\hat{a}$ and $\hat{b}$ (Theorem \ref{thm-tlk}), and the vectors $\vec{lk}(\hat{c}) \coloneqq (\mathrm{Lk}_{3,1}(\hat{c}), \mathrm{Lk}_{1,2}(\hat{c}),\mathrm{Lk}_{2,3}(\hat{c}))^T$ $(c=a,b)$ are non-zero vectors by the assumption, $\vec{lk}(\hat{a})$ and $\vec{lk}(\hat{b})$ must be contained in $W^\perp=\{(x_1, x_2, x_3)^T \mid x_s=x_t\}$. Thus we have the required result.
\end{proof}

\begin{proof}[Proof of Corollary \ref{cor1-2}]
By taking the contraposition of Theorem \ref{thm7-2}, we have the required result. 
\end{proof}

\begin{example}
Examples of pairs of pure $3$-braids $(a,b)$ satisfying the conditions of Corollary \ref{cor1-2} are given as follows. 
Let $n_1$ and $n_2$ be non-zero integers with $n_1 \neq n_2$, and let $m$ be any non-zero integer. 
Then the pair $(a,b)$ defined as in Example \ref{eg4-4}, i.e.,
$a=\sigma_1^{2n_1} \sigma_2^{2n_2}$ and $b=\tilde{\Delta}^m$, satisfies the conditions of Corollary \ref{cor1-2}.  
\end{example}

\section{Invariants of torus-covering $T^2$-links of degree 3. \newline 
(II) Number of Fox $p$-colorings}\label{sec5}

In this section, we review quandle colorings; in particular, $p$-colorings \cite{CKS, EN, Fox,  Joyce}. 
We investigate the number of $p$-colorings of a 3-braid for several examples (Propositions \ref{prop5-2-28} and \ref{prop5-2}). In this paper, we assume that $p$ is an odd prime. 

\subsection{Quandles}

 A {\it quandle} is a set $X$ equipped with a binary operation $*: X \times X \to X$ satisfying the following axioms. 
 
 \begin{enumerate}
 \item (Idempotency)
 For any $x \in X$, $x*x=x$. 
 
 \item (Right invertibility)
 For any $y,z \in X$, there exists a unique $x \in X$ such that $x*y=z$. 
 
 \item (Right self-distributivity)
 For any $x,y,z \in X$, $(x*y)*z=(x*z)*(y*z)$. 
 \end{enumerate}
 When $X$ consists of a finite number of elements, $X$ is called a \textit{finite quandle}.

 We review quandle colorings. Let $X$ be a finite quandle. 
 Let $L$ be a classical link and let $F$ be a surface-link, respectively. 
Let $D$ be a diagram of $L$ or $F$, and let $B(D)$ be the set of arcs or sheets of $D$. 
An {\it $X$-coloring} of $D$ is a map $C: B(D) \to X$ satisfying the coloring rule around each crossing or double point curve as shown in Figure \ref{fig-quandle}. For an $X$-coloring $C$, the image of an arc or sheet by $C$ is called a {\it color}. We say that an $X$-coloring is {\it trivial} (respectively, {\it non-trivial}) if colors of arcs or sheets consist of a single color (respectively, at least two distinct colors). 

\begin{figure}[ht]
\includegraphics[height=4cm]{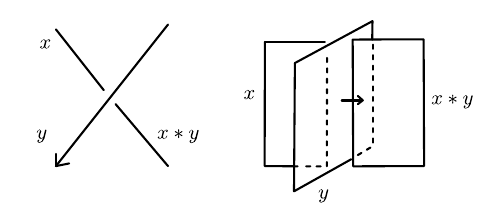}
\caption{The quandle coloring rule, where $x,y \in X$. We present the orientation of the over-sheet by its normal vector. The orientation of under-arcs or under-sheets is arbitrary.}\label{fig-quandle}
\end{figure}

\subsection{Dihedral quandles}
Let $N\geq 0$ be an integer. A {\it dihedral quandle} $R_N$ is given by the set $R_N=\mathbb{Z}/N\mathbb{Z}$ with the binary operation 
\[
x*y=2y-x,
\]
where $x,y \in R_N$. 
For a dihedral quandle $R_N$ ($N \neq 0$), 
we call an $R_N$-coloring an {\it $N$-coloring}. Let $L$ and $F$ be a classical link and a surface-link, respectively. Let $D$ be a diagram. 
We denote by $\mathrm{Col}_N(D)$ the set of $N$-colorings of $D$. 
We remark that $\mathrm{Col}_N(D)$ is a finite set. 
We denote by $\#\mathrm{Col}_N(L)$ or $\# \mathrm{Col}_N(F)$ the number of elements of $\mathrm{Col}_N(D)$, which is invariant under diagrams of $L$ or $F$.
In this paper, when we consider $N$-colorings for a diagram $D$ of a torus-covering $T^2$-link $\mathcal{S}_n(a,b)$, we take $D$ as the one given by the projection $\mathbb{R}^4 \to \mathbb{R}^3$, $(x,y,z,t) \mapsto (x,y,z)$, and we denote $D$ by the same notation $\mathcal{S}_n(a,b)$. 

\subsection{Number of Fox $p$-colorings of $\mathcal{S}_n(a,b)$}

For an $n$-braid $b$, let $A_b: (R_0)^n \to (R_0)^n$ be the map determined by $A_b=A_{b_1} \circ A_{\sigma_i^\epsilon}$ for a presentation $b=\sigma_i^\epsilon b_1$ $(i \in \{1, \ldots, n-1\}, \epsilon \in \{+1, -1\})$, where $A_{\sigma_i^\epsilon}$ is given by 
\begin{eqnarray*}
&&A_{\sigma_i}(x_1, \ldots, x_n)=(x_1, \ldots, x_{i-1}, x_{i+1}, x_i * x_{i+1}, x_{i+2}, \ldots, x_n), \\
&&A_{\sigma_i^{-1}}(x_1, \ldots, x_n)=(x_1, \ldots, x_{i-1}, x_{i+1}*x_i, x_i , x_{i+2}, \ldots, x_n). 
\end{eqnarray*}
We remark that $A_b$ is well-defined and bijective. We denote by the same notation $A_b$ the representation matrix in $M(n; \mathbb{Z})$ determined by $\mathbf{x} \mapsto A_b \, \mathbf{x}$ for a column vector $\mathbf{x}=(x_1, \ldots, x_n)^T$. The matrix $A_{b_0 b_1}$ for $n$-braids $b_0$ and $b_1$ satisfies $A_{b_0 b_1}=A_{b_1}A_{b_0}$. 
By taking composition with the projection $\mathbb{Z}^n \to (\mathbb{Z}/ N\mathbb{Z})^n$, 
the map $A_b$ induces a bijection $A_b \pmod N: (R_N)^n \to (R_N)^n$; note that for any $x,y \in \mathbb{Z}$ and $[x] \coloneqq x \pmod N \in \mathbb{Z}/N\mathbb{Z}$, $[2y-x]=2[y]-[x]$. 
For each $N$-coloring of $b$, the induced map $A_b \pmod N$ sends the $n$-tuple of the colors of the initial arcs of $b$ to that of the terminal arcs of $b$. 

When the closure of an $n$-braid $a$ has an $N$-coloring such that the $n$-tuple of the initial arcs is assigned with an $n$-tuple of colors $\mathbf{x} \in (\mathbb{Z}/N\mathbb{Z})^n$, it is a solution of a system of linear equations $A_a \,\mathbf{x}=\mathbf{x}$ in $(\mathbb{Z}/N\mathbb{Z})^n$. 
Note that if the number of $N$-colorings of a link is $N$, then it admits only trivial $N$-colorings. 
By construction of a torus-covering $T^2$-link, the following proposition is clear. We denote by $I$ the unit matrix. 
\begin{proposition}\label{newprop5-1}
Let $(a,b)$ be $n$-braids which commute. Then there is a natural bijection between 
the set of $N$-colorings of $\mathcal{S}_n(a,b)$ and $\mathrm{ker}(A_a-I \pmod{N} )\cap \mathrm{ker}(A_b-I \pmod{N})$, where $A_c-I \pmod{N}: (\mathbb{Z}/N\mathbb{Z})^n \to (\mathbb{Z}/N\mathbb{Z})^n$ for $c=a,b$. 
\end{proposition}

We recall that in this paper, $p$ is an odd prime. 
\begin{proposition}\label{cor1228}
Let $a$ be an $n$-braid, and let $r$ be the rank of $A_{a}-I \pmod{p}$. Then $\#\mathrm{Col}_p(\hat{a})=p^{n-r}$. 
\end{proposition}

In \cite{N5}, we showed the following

\begin{proposition}[{\cite[Lemma 6.3]{N5}}]\label{claim-1}
For a full twist $\tilde{\Delta}$ of $n$ strands, $A_{\tilde{\Delta}^2} \equiv I \pmod{p}$  
if $n$ is odd, and $A_{\tilde{\Delta}^p} \equiv I \pmod{p}$ if $n$ is even. 
\end{proposition}

So we have the following proposition.

\begin{proposition}\label{prop1222}
Let $(a,b)$ be $n$-braids which commute. Let $F=\mathcal{S}_n(a,b)$. 
If $A_b \equiv I \pmod{p}$, then $\#\mathrm{Col}_p(F)=\#\mathrm{Col}_p(\hat{a})$. 
In particular, for an arbitrary $n$-braid $a$ and an odd (respectively, even) integer $n$, and any integer $m$, if $F=\mathcal{S}_n(a,\tilde{\Delta}^{2m})$ (respectively, $\mathcal{S}_n(a,\tilde{\Delta}^{pm})$), then $\#\mathrm{Col}_p(F)=\#\mathrm{Col}_p(\hat{a})$. 
\end{proposition}

\subsection{The degree $3$ case}\label{sec5-4}
In this subsection, we consider 3-braids. 
Applying Proposition \ref{prop1222} to torus-covering $T^2$-links of degree 3, we have the following
\begin{corollary}\label{cor1222a}
Let $a$ be a 3-braid, and let $r$ be the rank of $A_{a}-I \pmod{p}$. 
Let $m$ be any integer. Then $\#\mathrm{Col}_p(\mathcal{S}_3(a,\tilde{\Delta}^{2m}))=p^{3-r}$. 
\end{corollary}

\begin{proof}
Since $A_{\tilde{\Delta}^2} \equiv I \pmod{p}$, 
Proposition \ref{prop1222} and Proposition \ref{cor1228} imply the required result. 
\end{proof}

We determine when a 3-braid $a$ satisfies $A_a \equiv I \pmod{p}$ for some cases. Proposition \ref{prop5-2-28} is an extended result of Proposition \ref{claim-1} for the degree 3 case. 
Let $n$ be an integer. 

\begin{proposition}\label{prop5-2-28}
We consider a 3-braid $(\sigma_1\sigma_2)^n$. Then $A_{(\sigma_1\sigma_2)^n} \equiv I \pmod{p}$ if and only if $n \equiv 0 \pmod{6}$. 
If $A^n \not\equiv I \pmod{p}$, then the rank of $A_{(\sigma_1 \sigma_2)^n}-I \pmod{p}$ is one if $p=3$ and $n \equiv 2 \pmod{6}$, or $p=3$ and $n \equiv 4 \pmod{6}$, and the rank of $A_{(\sigma_1 \sigma_2)^n}-I \pmod{p}$ is two otherwise. 
In terms of the number of $p$-colorings, 
\[
\# \mathrm{Col}_p(((\sigma_1\sigma_2)^n)^\wedge)=\begin{cases} p^3 & \text{if $n \equiv 0 \pmod{6}$} \\
p^2 & \text{if $p=3$, and $n \equiv 2$ or $n \equiv 4  \pmod{6}$} \\
p & \text{otherwise}. 
\end{cases}
\]
\end{proposition}

\begin{proof}
Put $A \coloneqq A_{\sigma_1 \sigma_2}$. 
Since 
\[
A_{\sigma_1}=\begin{pmatrix} 0 & 1 & 0\\
-1 & 2 & 0\\
0 & 0 & 1
\end{pmatrix}, \  A_{\sigma_2}=\begin{pmatrix} 1 & 0 & 0 \\
0 & 0 & 1 \\
0 & -1 & 2 \\
\end{pmatrix}, 
\] 
we calculate 
\begin{eqnarray*}
&&
A=A_{\sigma_1 \sigma_2}=A_{\sigma_2}A_{\sigma_1}=
\begin{pmatrix}
0 & 1 & 0\\
0 & 0 & 1\\
1 & -2& 2
\end{pmatrix}, \end{eqnarray*}
and \begin{eqnarray*}A-I=\begin{pmatrix}
-1 & 1 & 0\\
0 & -1 & 1\\
1 & -2& 1
\end{pmatrix} \to \begin{pmatrix}
1 & -1 & 0\\
0 & 1 & -1\\
0 & 0& 0
\end{pmatrix},   
\end{eqnarray*}
where $\to$ denotes row transformations; hence $\mathrm{rank} (A-I)=2$. 
We calculate 
\begin{eqnarray*}
&
A^2=
\begin{pmatrix}
0 & 0 & 1\\
1 & -2 & 2\\
2 & -3& 2
\end{pmatrix},  & A^2-I=\begin{pmatrix}
-1 & 0 & 1\\
1 & -3 & 2\\
2 & -3& 1
\end{pmatrix} \to \begin{pmatrix}
1 & 0 & -1\\
0 & 3 & -3\\
0 & 0& 0
\end{pmatrix},  \\
\\
  &
A^3=
\begin{pmatrix}
1 & -2 & 2\\
2 & -3 & 2\\
2 & -2& 1
\end{pmatrix},  & A^3-I=\begin{pmatrix}
0 & -2 & 2\\
2 & -3 & 2\\
2 & -2& 0
\end{pmatrix} \to \begin{pmatrix}
2 & -2 & 0\\
0 & 2 & -2\\
0 & 0& 0
\end{pmatrix},  \\
\\
 &
A^4=
\begin{pmatrix}
2 & -3 & 2\\
2 & -2 & 1\\
1 & 0& 0
\end{pmatrix},  & A^4-I=\begin{pmatrix}
1 & -3 & 2\\
2 & -3 & 1\\
1 & 0& -1
\end{pmatrix} \to \begin{pmatrix}
1 & 0 & -1\\
0 & 3 & -3\\
0 & 0& 0
\end{pmatrix},  \\
\\
 &
A^5=
\begin{pmatrix}
2 & -2 & 1\\
1 & 0 & 0\\
0 & 1& 0
\end{pmatrix},  & A^5-I=\begin{pmatrix}
1 & -2 & 1\\
1 & -1 & 0\\
0 & 1& -1
\end{pmatrix} \to \begin{pmatrix}
1 & -1 & 0\\
0 & 1 & -1\\
0 & 0& 0
\end{pmatrix}, \\
\\
& A^6=\begin{pmatrix}
1 & 0 & 0\\
0 & 1 & 0\\
0 & 0& 1
\end{pmatrix}=I. & 
\end{eqnarray*}
Hence we have the required result. 
\end{proof}

Next we consider $(\sigma_1\sigma_2^{-1})^n$. In this case we consider
\begin{align*}
\mathbb{Z}\Big[\frac{1+\sqrt{5}}{2}\Big] = \Bigg\{x+ \frac{1+\sqrt{5}}{2} y \,\Bigg|\, x,y \in \mathbb{Z}\Bigg\}, 
\end{align*}
the ring generated by $\frac{1+\sqrt{5}}{2}$ over $\mathbb{Z}$. 
For $\alpha_1, \alpha_2\in \mathbb{Z}[\frac{1+\sqrt{5}}{2}]$, we write $\alpha_1\equiv \alpha_2 \pmod{p}$ if $\alpha_1-\alpha_2 \in p\mathbb{Z}[\frac{1+\sqrt{5}}{2}]$. In other words, for $x,x', y,y' \in \mathbb{Z}$, we have $x+\frac{1+\sqrt{5}}{2}y  \equiv x'+\frac{1+\sqrt{5}}{2}y' \pmod{p}$ if and only if $x \equiv x' \pmod{p}$ and $y \equiv y' \pmod{p}$.

\begin{proposition}\label{prop5-2}
We consider a 3-braid $(\sigma_1\sigma_2^{-1})^n$. Then $A_{(\sigma_1\sigma_2^{-1})^n} \equiv I \pmod{p}$ if and only if $(\frac{3+\sqrt{5}}{2})^n \equiv 1 \pmod{p}$ in $\mathbb{Z}[\frac{1+\sqrt{5}}{2}]$. 
In terms of the number of $p$-colorings, 
\[
\# \mathrm{Col}_p(((\sigma_1\sigma_2^{-1})^n)^\wedge)=p^3 
\]
if and only if $(\frac{3+\sqrt{5}}{2})^n \equiv 1 \pmod{p}$  in $\mathbb{Z}[\frac{1+\sqrt{5}}{2}]$. 
In particular, 
\begin{align*}
A_{(\sigma_1\sigma_2^{-1})^n} & \equiv I \pmod{3} \text{ if and only if $n\equiv 0 \pmod{4}$}.  
\end{align*}
Moreover,  
if $n \not\equiv 0 \pmod{4}$, then the rank of $A_{(\sigma_1 \sigma_2^{-1})^n}-I \pmod{3}$ is two. In terms of the number of $p$-colorings, 
\[
\# \mathrm{Col}_3(((\sigma_1\sigma_2^{-1})^n)^\wedge)=\begin{cases} 27 & \text{if 
$n \equiv 0 \pmod{4}$} \\
3 & \text{otherwise}. 
\end{cases}
\]
\end{proposition}
\begin{proof}
Put $A \coloneqq A_{\sigma_1 \sigma_2^{-1}}$. 
Recall that 
\[
A_{\sigma_1}=\begin{pmatrix} 0 & 1 & 0\\
-1 & 2 & 0\\
0 & 0 & 1
\end{pmatrix}, \ A_{\sigma_2^{-1}}=\begin{pmatrix} 1 & 0 & 0 \\
0 & 2 & -1 \\
0 & 1 & 0 \\
\end{pmatrix}. 
\]
Hence we have
\[
A=A_{\sigma_1\sigma_2^{-1}}=A_{\sigma_2^{-1}}A_{\sigma_1}=\begin{pmatrix}
0 & 1 & 0\\
-2 & 4 & -1\\
-1 & 2 & 0
\end{pmatrix}. 
\]
We see that the eigenvalues of $A$ are $1, \frac{3 \pm\sqrt{5}}{2}$, and that $(1,\, 1,\, 1)^T$, $( 1, \, \frac{3\pm\sqrt{5}}{2}, \, \frac{1\pm\sqrt{5}}{2} )^T$ are the eigenvectors of $A$ with eigenvalues $1, \frac{3 \pm\sqrt{5}}{2}$, respectively. 
Put $\varepsilon=\frac{1+  \sqrt{5}}{2}$. Note that $\frac{3+\sqrt{5}}{2}=\varepsilon^2$. 

Define a homomorphism 
\begin{align*}
  f\colon \mathbb{Z}^3 \rightarrow \mathbb{Z}\oplus \mathbb{Z}[\varepsilon]
\end{align*}
of $\mathbb{Z}$-modules 
by $f(x,y,z)=(x+y+z, x+\varepsilon^2 y+\varepsilon z)$. Then we easily see that this is an isomorphism. Indeed, the inverse map is given by $f^{-1}(x,y+z\varepsilon) = (x-z,-x+y+z,x-y)$. 

Let $n\in \mathbb{Z}$ be any integer. Then, since 
\[
A^n  \begin{pmatrix} 1 \\ 1 \\ 1 \end{pmatrix}=\begin{pmatrix} 1 \\ 1 \\ 1 \end{pmatrix}, 
A^n  \begin{pmatrix} 1  \\ \varepsilon^2  \\ \varepsilon  \end{pmatrix}=\varepsilon^{2n} \begin{pmatrix} 1 \\ \varepsilon^2 \\ \varepsilon \end{pmatrix},
\]
we have the following commutative diagram: 
\[
\begin{tikzcd}
\mathbb{Z}^3 \arrow[r,"\times A^n"] \arrow[d, "\cong", "f"'] & \mathbb{Z}^3 \arrow[d,"f"', "\cong"]\\
\mathbb{Z}\oplus \mathbb{Z}[\varepsilon] \arrow[r, "\mathrm{id}\oplus \varepsilon^{2n}"]& \mathbb{Z}\oplus \mathbb{Z}[\varepsilon]. 
\end{tikzcd}
\]
Here, 
$\times A^n$ maps $(x,y,z) \in \mathbb{Z}^3$ to $(x,y,z)A^n$ and $\mathrm{id}\oplus \varepsilon^{2n}$ maps $(x,\alpha) \in \mathbb{Z}\oplus \mathbb{Z}[\varepsilon]$ to $(x,\varepsilon^{2n}\alpha)$. 
Then by taking modulo $p$, we have the following commutative diagram: 
\[
\begin{tikzcd}[column sep=huge]
(\mathbb{Z}/p\mathbb{Z})^3 \arrow[r,"\times A^n \bmod{p}"] \arrow[d, "\cong", "f \bmod{p}"'] & (\mathbb{Z}/p\mathbb{Z})^3 \arrow[d,"f \bmod{p}"', "\cong"]\\
\mathbb{Z}/p\mathbb{Z}\oplus \mathbb{Z}[\varepsilon]/p\mathbb{Z}[\varepsilon] \arrow[r, "\mathrm{id}\oplus \varepsilon^{2n} \bmod{p}"]& \mathbb{Z}/p\mathbb{Z}\oplus \mathbb{Z}[\varepsilon]/p\mathbb{Z}[\varepsilon]. 
\end{tikzcd}
\]
 Hence we find $A^n \equiv I \pmod{p}$ if and only if $\varepsilon^{2n} \equiv 1 \pmod{p}$ (in $\mathbb{Z}[\varepsilon]$). The interpretation in terms $\#\mathrm{Col}_p$ follows from Proposition \ref{cor1228}. 
 
In the case $p=3$, a direct computation shows that $\varepsilon^{2n} \equiv 1 \pmod{3}$ if and only if $n \in 4 \mathbb{Z}$. 
Furthermore, we compute 
\[
\mathbb{Z}[\varepsilon]/3\mathbb{Z}[\varepsilon]=\mathbb{Z}[X]/(X^2-X-1) \otimes \mathbb{Z}/3\mathbb{Z}=(\mathbb{Z}/3\mathbb{Z})[X]/(X^2-X-1).
\] 
Then since $5$ is a quadratic non-residue mod $3$, $X^2-X-1$ is an irreducible polynomial in $(\mathbb{Z}/3\mathbb{Z})[X]$, and hence $\mathbb{Z}[\varepsilon]/3\mathbb{Z}[\varepsilon]$ turns out to be the finite field $\mathbb{F}_9$ of order $9$. 
 Therefore, we have
 \[
 \mathrm{ker}(\varepsilon^{2n}-1 \bmod{3}: \mathbb{F}_9 \to \mathbb{F}_9)=
 \begin{cases}
 \mathbb{F}_9 & \text{if $\varepsilon^{2n}-1 \equiv 0 \pmod{3}$} \\
 0 & \text{otherwise}, 
 \end{cases}
 \]
 and hence we find
 \begin{align*}
   \mathrm{rank}(A^n-I \bmod{3}) =  \begin{cases}
 0 & \text{if $n \equiv 0 \pmod{4}$} \\
 2 & \text{otherwise}. 
 \end{cases}
 \end{align*}
The interpretation in terms $\#\mathrm{Col}_p$ follows again from Proposition \ref{cor1228}.
\end{proof}

\begin{remark}\label{rmk:generalization}
\begin{enumerate}
\item 
It is possible to extend the statement of Proposition \ref{prop5-2} for $p=3$ to any prime using the ray class numbers of the corresponding quadratic field. 
For instance, assume $p \neq 2$ and set
\begin{align*}
  \varphi_p = 
  \begin{cases}
    (p-1)^2 & \text{ if $p \equiv 1,4 \pmod{5}$}\\
    p^2-1 & \text{ if $p \equiv 2,3 \pmod{5}$}\\
    p(p-1) & \text{ if $p=5$}, 
  \end{cases}
\end{align*}
and let $h_{p,\infty}$ denote the ray class number of $\mathbb{Z}[\frac{1+\sqrt{5}}{2}]$ modulo $(p,\infty)$, i.e., $h_{p,\infty}$ is the order of the ray class group $\mathrm{Cl}_{p,\infty}(\mathbb{Z}[\tfrac{1+\sqrt{5}}{2}])$ of $\mathbb{Z}[\tfrac{1+\sqrt{5}}{2}]$ modulo $(p,\infty)$ (see \cite[p.33, Definition 5.4]{KL}). Then we can show $h_{p,\infty} \mid \varphi_p$ and 
\begin{align*}
  \# \mathrm{Col}_p(((\sigma_1\sigma_2^{-1})^n)^\wedge)=
\begin{cases} 
p^3 & \text{if $n \equiv 0 \pmod{\frac{\varphi_p}{h_{p,\infty}}}$} \\
p^2 & \text{if $p =5$ and $n \equiv 4,6 \pmod{10}$} \\
p & \text{otherwise}. \\
\end{cases}
\end{align*}
Indeed, this follows from a similar argument as in the case $p=3$, together with the exact sequence
\begin{align*}
 \varepsilon^{2\mathbb{Z}}
\rightarrow
\Big(\mathbb{Z}[\varepsilon]/p\mathbb{Z}[\varepsilon]\Big)^{\times} 
\rightarrow 
\mathrm{Cl}_{p,\infty}(\mathbb{Z}[\varepsilon])
\rightarrow 1, 
\end{align*}
where $\varepsilon = \tfrac{1+\sqrt{5}}{2}$ and $\varepsilon^{2\mathbb{Z}}$ is the subgroup of $\mathbb{Z}[\varepsilon]^{\times}$ generated by $\varepsilon^2$ (see \cite[p.42, Theorem 6.5]{KL}). 
\item 
The argument in Proposition \ref{prop5-2} and Remark \ref{rmk:generalization} (1) applies to an arbitrary $3$-braid $b$ and we can compute the rank of $A_{b^n}-I$ in terms of the corresponding unit in a quadratic extension of $\mathbb{Z}$. 
In particular, Proposition \ref{prop5-2-28} can also be proved in a similar way to Proposition \ref{prop5-2}. 
It might be interesting to investigate the applications of such an arithmetic interpretation of $p$-colorings to the study of braids $b$ or to the torus covering $T^2$-links. 
\end{enumerate}
\end{remark}

\section{Invariants of torus-covering $T^2$-links of degree 3. \newline 
(III) Quandle cocycle invariant associated with $p$-colorings}\label{sec6}
In Section \ref{sec6-1}, we review the quandle cocycle invariant associated with $p$-colorings \cite{CJKLS, CKS}. 
In Section \ref{0306-new-sec6-2}, we define the reduced quandle cocycle invariant (Definition \ref{reduced}) and prove Theorem \ref{thm0304}  and  Corollary \ref{thm1-3}. 
In Section \ref{sec6-2}, we focus on tri-colorings and classify $\mathcal{S}_3(a,b)$ under the qdl-equivalence relation, which is invariant under the quandle cocycle invariant (Theorem \ref{thm6-9}), and then we prove Theorem \ref{thm1-4}. 

\subsection{Quandle cocycle invariant}\label{sec6-1}
Let $X$ be a finite quandle, and 
let $G$ be an abelian group. A \textit{3-cocycle} is a map $f: X \times X \times X \to G$ satisfying the following conditions: 
\begin{eqnarray*}
&& \bullet \ f(s,t,u)+f(s*u, t*u, v)+f(s,u,v)
\\
&& \quad \quad =
f(s*t, u, v)+f(s,t, v)+f(s*v, t*v, u*v), \\
&& \bullet \ f(s,s,t)=0, \\
&& \bullet \ f(s,t,t)=0, 
\end{eqnarray*}
for any $s,t,u,v \in X$. 

For an $X$-coloring $C$ of a diagram $D$ of a surface-link $F$, at each triple point $\tau$ of $D$, we define the \textit{weight} $W_f(\tau; C)$ at $\tau$ for a 3-cocycle $f$ by $W_f(\tau; C)=f(x,y,z)$ (respectively, $-f(x,y,z)$) if $\tau$ is a positive (respectively, negative) triple point, where $x, y, z$ are the colors of sheets as in Figure \ref{fig2}. 
We denote by $X_3(D)$ the set of triple points of $D$. Put 
\[
\Phi_f(F; C)= \sum_{\tau\in X_3(D)} W_f(\tau; C). 
\]
It is known that $\Phi_f(F; C)$ is invariant under Roseman moves for diagrams colored by $X$. We call $\Phi_f(F; C)$ the {\it quandle cocycle invariant} of $F$ associated with an $X$-coloring $C$ and a 3-cocycle $f$. 
Since we consider a finite quandle $X$, the set of sheets $B(D)$ is a finite set, so $\mathrm{Col}_X(D)$ consists of a finite number of elements. We define the \textit{ quandle cocycle invariant} of $F$ associated with a 3-cocycle $f$ by the multi-set  
\[
\Phi_f(F)=\{ \Phi_f(F; C) \mid C \in \mathrm{Col}_X(D)\}. 
\]
\begin{figure}[ht]
\includegraphics[height=4.5cm]{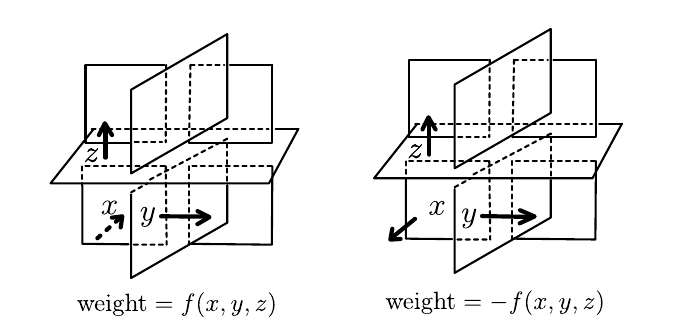}
\caption{The weight at a triple point, where $x$, $y$ and $z$ are the colors by an $X$-coloring $C$, and $f$ is a 3-cocycle.}\label{fig2}
\end{figure}
By definition, the quandle cocycle invariant for a surface-link $F$ associated with a 3-cocycle $f$ satisfies $\Phi_f(-F^*)=-\Phi_f(F)$, where $-\Phi_f(F)$ is the multi-set obtained from $\Phi_f(F)$ by replacing each element with its inverse. This relation is useful in showing that a surface-link is not $(-)$-amphicheiral. \\

The quandle cocycle invariant of $\mathcal{S}_n(a, \tilde{\Delta}^m)$ is calculated using the shadow cocycle invariants of the closed braid $\hat{a}$. We review the shadow cocycle invariant of a classical link $L$. 
Let $C$ be an $X$-coloring of a diagram $D$ of $L$ associated with a generic projection $\pi$. 
Then a {\it region} associated with $D$ is defined as a connected component of the complement of the image $\pi(L)$. We recall that we denote by $B(D)$ the set of arcs of $D$. We denote by $B^*(D)$ the the union of $B(D)$ and the set of regions of $\mathbb{R}^2$ associated with $D$. For $x \in X$, let $C_x^*: B^*(D) \to X$ be a map satisfying the following conditions. 

\begin{itemize}
\item
The color of the unbounded region is $x$. 
\item
The restriction of $C^*_x$ to $B(D)$ coincides with $C$. 
\item
Around each crossing, the regions are assigned with colors as in Figure \ref{fig3}.
\end{itemize}
Given $C$ and $x$, the map $C^*_x$ exists uniquely. 
We call the color of the unbounded region the {\it base color}. 
For a 3-cocycle $f$ and $C$ and $x$, we define the \textit{weight} $W_f^*(\tau; C, x)$ at a crossing $\tau$ as in Figure \ref{fig3}. 
We denote by $X_2(D)$ the set of crossings of $D$. 
We define 
\[
\Psi_f^*(L; C, x)=\sum_{\tau\in X_2(D)} W_f^*(\tau; C, x).
\]
It is known that $\Psi_f^*(L; C, x)$ is invariant under Reidemeister moves for diagrams colored by $X$. We call $\Psi_f^*(L; C, x)$ the \textit{shadow cocycle invariant} of $L$ with the base color $x$ associated with an $X$-coloring $C$ and a 3-cocycle $f$.

\begin{figure}[ht]
\includegraphics[height=4.5cm]{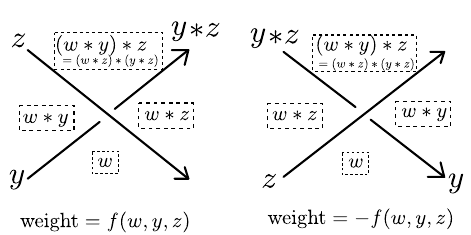}
\caption{A shadow coloring and the weight at a crossing, where $y$, $z$ and $w$ are the colors by $C^*_x$, and $f$ is a 3-cocycle.}\label{fig3}
\end{figure}

For the dihedral quandle $R_p$, it is known \cite{Mochizuki} that for any odd prime $p$, 3-cocycles for $R_p$ with the coefficient group $\mathbb{Z}/p\mathbb{Z}$ form a cyclic group with order $p$, with a generator $\theta_p: R_p \times R_p \times R_p \to \mathbb{Z}/p\mathbb{Z}$ given by
\[
\theta_p(s,t,u)=\frac{(s-t)((2u-t)^p+t^p-2u^p)}{p}. 
\]
We call $\theta_p$ the \textit{Mochizuki 3-cocycle}, and we denote the quandle cocycle invariant and the shadow cocycle invariant associated with $\theta_p$ by $\Phi_p(F)$ and $\Psi_p^*(L;C, x)$, respectively.

\begin{theorem}[{\cite[Theorem 7.1]{N5}}]\label{thm6-1}
Let $a$ be an $n$-braid and let $m$ be an integer.  Assume that $(A_{\tilde{\Delta}})^m \equiv I \pmod{p}$. Then 
\[
\Phi_p(\mathcal{S}_n(a, \tilde{\Delta}^m))=\{-mn\Psi_p^*(\hat{a}; C, 0) \mid C\in \mathrm{Col}_p(\hat{a})\}. 
\]
In particular, when $n=3$, 
\begin{eqnarray*}
\Phi_p(\mathcal{S}_3(a, \tilde{\Delta}^{m})) = \left\{  -3m\Psi_p^*(\hat{a}; C, 0) \mid C \in \mathrm{Col}_p(\hat{a}) \right\}. \end{eqnarray*}
\end{theorem}

\begin{theorem}[{\cite[Theorem 7.2]{N5}}, see also \cite{AS}]\label{thm6-2}
Let $a$ be an $n$-braid presented by 
\[
a=\prod_{j=1}^N \sigma_1^{pk_{1,j}} \sigma_2^{pk_{2,j} } \cdots \sigma_{n-1}^{pk_{n-1,j} }    
\]
for some integer $N>0$ and $k_{1,1}, \ldots, k_{n-1,N} \in \mathbb{Z}$, and let $\nu_i=\sum_{j=1}^N k_{i,j}$ $(i=1, \ldots, n-1)$. 
Let $m$ be any integer. Then, when $n$ is odd, 
\[
\Phi_p(\mathcal{S}_n(a, \tilde{\Delta}^{2m}))=\left\{ \left. \underbrace{2mn \sum_{i=1}^{n-1} \nu_i x_i^2, \ldots, 2mn \sum_{i=1}^{n-1} \nu_i x_i^2}_{p} \,\right\vert \, x_1, \ldots, x_{n-1} \in \mathbb{Z}/p\mathbb{Z}\right\}, 
\]
and when $n$ is even, 
\[
\Phi_p(\mathcal{S}_n(a, \tilde{\Delta}^{pm}))=\{  \underbrace{0, \ldots, 0}_{p^n} \}.  
\]
In particular, when $n=3$, 
\begin{eqnarray*}
\Phi_p(\mathcal{S}_3(a, \tilde{\Delta}^{2m})) =
\{  \underbrace{6m(\nu_1 x_1^2+\nu_2 x_2^2), \ldots, 6m(\nu_1 x_1^2+\nu_2 x_2^2)}_{p} \mid  x_1, x_2 \in \mathbb{Z}/p\mathbb{Z}\}.
\end{eqnarray*}
\end{theorem}

\vspace{0.5cm}
We use Proposition \ref{claim-1} to show Theorem \ref{thm6-2}. 

Now, we give a class of 3-braids which has the same quandle cocycle invariant for $p$-colorings. 
For a 3-braid $a$ 
with a presentation $a=\sigma_{1}^{n_1 } \sigma_{2}^{n_2}\sigma_{1}^{n_3 } \cdots \sigma_2^{n_k}$ $(n_1, \ldots, n_k \in \mathbb{Z})$, we define the set $[a]$ of 3-braids associated with $p$ by  
\[
[a]=\{ \sigma_{1}^{n_1'} \sigma_{2}^{n_2'}\sigma_{1}^{n_3'} \cdots \sigma_2^{n_k'} \mid n_i' \equiv n_i \ \mathrm{mod}\ {p} \ (i=1,2,\ldots, k)\}. 
\]

\begin{proposition}\label{prop2026-6-3}
For an arbitrary 3-braid $a$ and any integer $m$, we have the following. 
\begin{enumerate}
\item[$(1)$]
$\# \mathrm{Col}_p(\hat{a})=\#\mathrm{Col}_p(\hat{a_1})$ for any 3-braid $a_1 \in [a]$. 
\vspace{0.2cm}
\item[$(2)$]
$\Phi_p(\mathcal{S}_3(a,\tilde{\Delta}^{2pm}))= \{ \underbrace{  0, \ldots, 0 }_{\# \mathrm{Col}_p(\hat{a})} \}$. 
\vspace{0.2cm}
\item[$(3)$]
$\Phi_p(\mathcal{S}_3(a,\tilde{\Delta}^{2pm}))=\Phi_p(\mathcal{S}_3(a_1, \tilde{\Delta}^{2pm}))$ for any $a_1 \in [a]$.
\end{enumerate}

\end{proposition}

\begin{proof}
Since $A_{\sigma_i^p} \equiv I \pmod{p}$ for $i=1,2$, the numbers of $p$-colorings of $\hat{a}$ and $\hat{a_1}$ coincide for any $a_1 \in [a]$; thus we have (1). 
Since $A_{\tilde{\Delta}^{2}} \equiv I \pmod{p}$ by Proposition \ref{claim-1}, Theorem \ref{thm6-1} implies 
(2). The equation (3) is the result of (1) and (2). 
\end{proof}

\subsection{Proof of Theorem \ref{thm0304} and 
 Corollary \ref{thm1-3}}\label{0306-new-sec6-2}

For $p$-colorings, we use another presentation of the quandle cocycle invariant of a surface-link $F$, which is given by 
\begin{equation}\label{eq0304}
\sum_{C \in \mathrm{Col}_p(D)} v^{\Phi_p(F; C)} \in \mathbb{Z}[v, v^{-1}]/(v^p-1), 
\end{equation}
where we use the notation in Subsection \ref{sec6-1}. 
We denote the quandle cocycle invariant in this form also 
 by the same notation $\Phi_p(F)$ or by $\Phi_p(F)(v)$. 
Moreover, it will be convenient to consider the following \textit{reduced quandle cocycle invariant} $\tilde{\Phi}_p$. 

Let $\zeta_p =e^{2\pi \sqrt{-1}/p}$ denote the primitive root of unity, and let 
\begin{align*}
  \mathbb{Z}[\zeta_p] = \Big\{\sum_{j=0}^{p-2}c_j \zeta_p^j  \,\Big|\, c_0, \dots, c_{p-2} \in \mathbb{Z}\Big\}
\end{align*}
be the subring of $\mathbb{C}$ generated by $\zeta_p$ over $\mathbb{Z}$. 
 
\begin{definition}\label{reduced}
We define the \textit{reduced quandle cocycle invariant} $\tilde{\Phi}_p \in \mathbb{Z}[\zeta_p]$ to be the value of $\Phi_p(F)(v) \in \mathbb{Z}[v,v^{-1}]/(v^p-1)$ at $v=\zeta_p$, that is, 
\begin{align*}
\tilde{\Phi}_p(F) \coloneqq \Phi_p(F)(\zeta_p) = \sum_{C \in \mathrm{Col}_p(D)} \zeta_p^{\Phi_p(F; C)} \in \mathbb{Z}[\zeta_p] \subset \mathbb{C}. 
\end{align*}
\end{definition}

\begin{lemma}\label{lem:complex conjugate}
We have
\begin{align*}
\Phi_p(-F^*)(v)=\Phi_p(F)(v^{-1}), \qquad 
\tilde{\Phi}_p(-F^*)=\overline{\tilde{\Phi}_p(F)}, 
\end{align*}
where $\overline{\phantom{z}}$ denotes the complex conjugation, e.g., $\overline{\zeta_p}=\zeta_p^{-1}=\zeta_p^{p-1}$. 
\end{lemma}
\begin{proof}
  This follows form the fact that $\Phi_p(-F^*)=-\Phi_p(F)$ as multi-sets. 
\end{proof}

Note that we can easily recover the original $\Phi_p(F)$ from $\tilde{\Phi}_p(F)$. 
Indeed, we have an injective ring homomorphism 
\begin{align*}
 \iota\colon  \mathbb{Z}[v,v^{-1}]/(v^p-1) \hookrightarrow \mathbb{Z} \times \mathbb{Z}[\zeta_p];\, 
P(v) \mapsto (P(1),P(\zeta_p)), 
\end{align*}
and the image of this map is
\begin{align*}
\mathrm{image}(\iota)=  \Big\{\Big(x,\sum_{j=0}^{p-2}c_j \zeta_p^j\Big) \in \mathbb{Z}\times \mathbb{Z}[\zeta_p] \,\Big|\, x-\sum_{j=0}^{p-2}c_j \equiv 0 \pmod{p}\Big\}. 
\end{align*}
The inverse image of $\Big(x,\sum_{j=0}^{p-2}c_j \zeta_p^j\Big) \in \mathrm{image}(\iota)$ is given by
\begin{align}\label{eqn:inverse of iota}
  \sum_{j=0}^{p-2}c_j v^j + \frac{x-\sum_{j=0}^{p-2}c_j}{p}\sum_{j=0}^{p-1}v^j. 
\end{align}
In the case of the quandle cocycle invariant $\Phi_p(F)$, we have $\Phi_p(F)(1) = \#\mathrm{Col}_p(D)$, and hence we can recover $\Phi_p(F)$ by applying \eqref{eqn:inverse of iota} to 
\begin{align*}
  (\#\mathrm{Col}_p(D), \tilde{\Phi}_p(F)) \in \mathrm{image}(\iota). 
\end{align*}

Recall that $\Big(\frac{~}{p}\Big)$ denotes the Legendre symbol, i.e., for $\nu \in \mathbb{Z}/p\mathbb{Z}$, we have
\begin{align*}
  \Big(\frac{\nu}{p}\Big)
=
\begin{dcases}
  1 &\text{ if $\nu$ is a quadratic residue mod ${p}$ and $\nu \not\equiv 0 \pmod{p}$} \\
  -1 &\text{ if $\nu$ is a quadratic non-residue mod ${p}$ and $\nu \not\equiv 0 \pmod{p}$} \\
  0 &\text{ if $\nu \equiv 0 \pmod{p}$}. 
\end{dcases}
\end{align*}
We briefly review some standard facts about the quadratic Gauss sum. 
\begin{definition}
For $\nu \in \mathbb{Z}/p\mathbb{Z}$ (and an odd prime number $p$), the quadratic Gauss sum $G(\nu,p)$ is defined as
\begin{align*}
  G(\nu,p) \coloneqq \sum_{j=0}^{p-1}\zeta_p^{\nu j^2}. 
\end{align*}
\end{definition}

\begin{proposition}[{\cite[pp.86--87]{Lang}}]\label{prop:gauss sum}
We have
\begin{align*}
  G(\nu,p)
=
  \begin{dcases}
  p &\text{ if $\nu \equiv 0 \pmod{p}$}  \\
  \Big(\frac{\nu}{p}\Big) \varepsilon_p \sqrt{p} &\text{ if $\nu \not\equiv 0 \pmod{p}$}, 
  \end{dcases}
\end{align*}
where 
\begin{align*}
  \varepsilon_p 
=
\begin{dcases}
  1 &\text{ if $p\equiv 1 \pmod{4}$}\\ 
  \sqrt{-1} &\text{ if $p\equiv 3 \pmod{4}$}. 
\end{dcases}
\end{align*}
\end{proposition}

This enables us to prove Theorem \ref{thm0304} and Corollary \ref{thm1-3}. 
In other words, we can compute the quandle cocycle invariant $\tilde{\Phi}_p(\mathcal{S}_n(a,\tilde{\Delta}^{2m}))$ for $n$-braids $a$ in Theorem \ref{thm6-2} and give a sufficient condition for $\mathcal{S}_n(a,\tilde{\Delta}^{2m})$ not to be $(-)$-amphicheiral.  
\begin{proof}[Proof of Theorem \ref{thm0304}]
Recall that we need to prove
 \begin{align*}
  \tilde{\Phi}_p(\mathcal{S}_n(a,\tilde{\Delta}^{2m}))
=
p^{n-\frac{1}{2}\#J}\varepsilon_p^{\#J}
\prod_{i \in J}\Big(\frac{2mn\nu_i}{p}\Big)
\end{align*}
for 
\begin{itemize}
\item $a=\prod_{j=1}^N \sigma_1^{pk_{1,j}} \sigma_2^{pk_{2,j} } \cdots \sigma_{n-1}^{pk_{n-1,j} }$, 
\item 
$\nu_i =\sum_{j=1}^N k_{i,j}$ $(i=1, \ldots, n-1)$, 
\item 
$J = \{ i \in\{1, \dots, n-1\} \mid 2mn \nu_i \not\equiv 0 \pmod{p}\}$.
\end{itemize}
By Theorem \ref{thm6-2} we have
\begin{align*}
\tilde{\Phi}_p(\mathcal{S}_n(a,\tilde{\Delta}^{2m}))
&=
p\sum_{(x_1,\dots,x_{n-1})\in (\mathbb{Z}/p\mathbb{Z})^{n-1}}\zeta_p^{2mn\sum_{i=1}^{n-1}\nu_i x_i^2} \\
&= 
p \prod_{i=1}^{n-1} G(2mn\nu_i,p). 
\end{align*}
Hence the theorem follows from Proposition \ref{prop:gauss sum}. 
\end{proof}

\begin{proof}[Proof of Corollary \ref{thm1-3}]
We need to show that
\begin{align*}
  \tilde{\Phi}_p(\mathcal{S}_n(a,\tilde{\Delta}^{2m})) \neq \overline{\tilde{\Phi}_p(\mathcal{S}_n(a,\tilde{\Delta}^{2m}))}
\end{align*}
if and only if $p\equiv 3 \pmod{4}$ and $\#J$ is odd, and that $\mathcal{S}_n(a,\tilde{\Delta}^{2m})$ is not $(-)$-amphicheiral if $p\equiv 3 \pmod{4}$ and $\#J$ is odd. 
By Theorem \ref{thm0304}, we see that $\tilde{\Phi}_p(\mathcal{S}_n(a,\tilde{\Delta}^{2m})) \in \mathbb{R}$ if and only if $p\equiv 1 \pmod{4}$ or $\#J$ is even. 
This shows the first part. 
The latter assertion then follows from Lemma \ref{lem:complex conjugate}. 
\end{proof}

As remarked earlier, we can recover $\Phi_p(\mathcal{S}_n(a,\tilde{\Delta}^{2m}))(v)$ from $\tilde{\Phi}_p(\mathcal{S}_n(a,\tilde{\Delta}^{2m}))$. For instance, in the case $n=3$ we have the following 
\begin{theorem}
Let the notation be the same as in Theorem \ref{thm0304} with $n=3$. Furthermore, assume $p\nmid 6m$. 
Then we have
\begin{align*}
\Phi_p(\mathcal{S}_3(a, \tilde{\Delta}^{2m}))(v) 
=
\begin{dcases}
p(2p-1)+p(p-1)\sum_{j=1}^{p-1}v^j
 &\text{if $\left(\frac{-\nu_1 \nu_2}{p}\right)=1$}\\
p+p(p+1)\sum_{j=1}^{p-1}v^j
 &\text{if $\left(\frac{-\nu_1 \nu_2}{p}\right)=-1$}\\
p^2\sum_{j=0}^{p-1} \Bigg(1+\Bigg(\frac{6m\nu_i j}{p}\Bigg)\Bigg) v^j
 &\text{if $p \mid \nu_1\nu_2$ and $p \nmid \nu_i$} \\
p^3
 &\text{if $p \mid \nu_1$ and $p\mid \nu_2$}. \\
\end{dcases}
\end{align*}
\end{theorem}
\begin{proof}
The cases where $p \mid \nu_1\nu_2$ follow directly from Theorem \ref{thm6-2} without using Theorem \ref{thm0304}. 

We consider the case $p \nmid \nu_1\nu_2$. We then have $\#J=2$. 
Furthermore, notice that $\varepsilon_p^2 = \Big(\frac{-1}{p}\Big)$ (see \cite[p.77]{Lang}). 
Hence by Theorem \ref{thm0304}, we find
\begin{align*}
\tilde{\Phi}_p(\mathcal{S}_3(a,\tilde{\Delta}^{2m}))
=
p^{2}
\Big(\frac{-\nu_1\nu_2}{p}\Big) \in \mathbb{Z}. 
\end{align*}
Now, note that by Theorem \ref{thm6-2}, we have $\#\mathrm{Col}_p(\mathcal{S}_3(a,\tilde{\Delta}^{2m}))=p^3$. 
Therefore, by \eqref{eqn:inverse of iota}, we obtain
\begin{align*}
\Phi_p(\mathcal{S}_3(a, \tilde{\Delta}^{2m})) (v)
=
p^{2}\Big(\frac{-\nu_1\nu_2}{p}\Big) + \frac{p^3-p^{2}\Big(\frac{-\nu_1\nu_2}{p}\Big)}{p} \sum_{j=0}^{p-1}v^j. 
\end{align*}
Thus we get the desired formula by setting $\Big(\frac{-\nu_1\nu_2}{p}\Big)=\pm 1$. 
\end{proof}

\subsection{Tri-colorings and the associated quandle cocycle invariant}\label{sec6-2}

For tri-colorings also we use the presentation of the quandle cocycle invariant of a surface-link $F$ in Subsection \ref{0306-new-sec6-2} (\ref{eq0304}).

\begin{proposition}\label{a-prop6-5}
For any 3-braid $a$, and any integer $m$, 
\[
\Phi_3(\mathcal{S}_3(a, \tilde{\Delta}^{m}))=\begin{cases} 3 & \text{if $m$ is odd} \\
\# \mathrm{Col}_3(\hat{a}) & \text{if $m$ is even.}
\end{cases} 
\]
\end{proposition}

\begin{proof}
Theorem \ref{thm6-1} and Proposition \ref{claim-1} imply the case when $m$ is even. 
Assume that $m$ is odd. By Proposition \ref{prop5-2-28}, 
we see that the number of tri-colorings for the closure of $\tilde{\Delta}^{m}=(\sigma_1\sigma_2)^{3m}$ is $3$: thus $\mathcal{S}_3(a, \tilde{\Delta}^{m})$ for any 3-braid $a$ admits only trivial tri-colorings. 
Hence the quandle cocycle invariant $\Phi_3(\mathcal{S}_3(a,\tilde{\Delta}^{m}))$ is $3$ for any 3-braid $a$ and any odd integer $m$. 
\end{proof}

\begin{proposition}\label{prop6-5}
Let $(a, b)$ be 3-braids which commute. Then $a=c^{l_1} \tilde{\Delta}^{m_1}$ and $b=c^{l_2} \tilde{\Delta}^{m_2}$ for some 3-braid $c$ and some integers $l_1, l_2, m_1, m_2$. 
\end{proposition}

\begin{proof}
The 3-braid group $B_3=\langle \sigma_1, \sigma_2 \mid \sigma_1 \sigma_2 \sigma_1=\sigma_2 \sigma_1 \sigma_2\rangle$ is isomorphic to the knot group of a trefoil and it has another presentation $G=\langle x,y \mid x^2=y^3 \rangle$. 
Since the center $Z(G)$ of $G$ is an infinite cyclic group generated by $z=x^2=y^3$, and 
$G/Z(G) \cong \mathbb{Z}/2\mathbb{Z} *\mathbb{Z}/3\mathbb{Z}$, 
we see that any pair $(a_1, b_1)$ of elements of $G$ satisfying $a_1b_1=b_1a_1$ is written as $a_1=c_1^{l_1} z^{m_1}$ and $b_1=c_1^{l_2} z^{m_2}$ for some $c_1 \in G/Z(G)$ and some integers $l_1, l_2, m_1, m_2$. Since the center of $B_3$ is an infinite cyclic group generated by $\tilde{\Delta}$, $\tilde{\Delta}$ corresponds to $z$ or $z^{-1}$ in $G$. Hence, interpreting $a_1$ and $b_1$ as elements in the 3-braid group $B_3$, we have the required result. 
\end{proof}

For a 3-braid $a$ 
with a presentation $a=\sigma_{1}^{n_1 } \sigma_{2}^{n_2}\sigma_{1}^{n_3 } \cdots \sigma_2^{n_k}$ $(n_1, \ldots, n_k \in \mathbb{Z})$, we consider the set $[a]$ of 3-braids associated with $3$, which is given by  
\[
[a]=\{ \sigma_{1}^{n_1'} \sigma_{2}^{n_2'}\sigma_{1}^{n_3'} \cdots \sigma_2^{n_k'} \mid n_i' \equiv n_i \ \mathrm{mod}\ {3} \ (i=1,2,\ldots, k)\}. 
\]

\begin{definition}\label{def6-6}
We say that two torus-covering $T^2$-links of degree 3 
are {\it qdl-equivalent} if they are related by $\sim$ and $\sim_{qdl}$, where $\sim$ is the equivalence relation as surface-links in $\mathbb{R}^4$ which include (\ref{E0})--(\ref{eq1121-3}) in Theorem \ref{thm3-1}, and $\sim_{qdl}$ is given as follows.  Let $m$ be any integer. 
\begin{align}
& \mathcal{S}_3(a,b) \sim_{qdl} \mathcal{S}_3(a, ab), \tag{Q1} \label{Q1}\\
&
\mathcal{S}_3(a\tilde{\Delta}^{\pm2},\tilde{\Delta}^{m})\sim_{qdl} \mathcal{S}_3(a, \tilde{\Delta}^{m}),
\tag{Q2}\label{Q2}
\\
&
\mathcal{S}_3(a,\tilde{\Delta}^{m})\sim_{qdl} \mathcal{S}_3(a, \tilde{\Delta}^{m\pm 2}), \tag{Q3}\label{Q3}\\
&
\mathcal{S}_3(a,\tilde{\Delta}^{m})\sim_{qdl} \mathcal{S}_3(a_1, \tilde{\Delta}^{m}) \text{ for any $a_1 \in [a]$}. \tag{Q4}\label{Q4}
\end{align}

\end{definition}

\begin{theorem}\label{thm6-6}
Let $(a,b)$ be 3-braids which commute. We denote by $\mathcal{C}(a,b)$ the qdl-equivalence class of $\mathcal{S}_3(a,b)$. 
Then, for any $F \in \mathcal{C}(a,b)$, the quandle cocycle invariant $\Phi_3(F)$ has the same value. 
\end{theorem}

\begin{proof}
It suffices to show that the quandle cocycle invariants are the same for the torus-covering $T^2$-links given in (\ref{Q1})--(\ref{Q4}). The case (\ref{Q3}) follows from Proposition \ref{a-prop6-5}. The case (\ref{Q4}) follows from  Propositions \ref{prop2026-6-3} and \ref{a-prop6-5}. 
Since $A_{\tilde{\Delta}^{\pm 2}} \equiv I \pmod{3}$ by Propositions \ref{claim-1} or \ref{prop5-2-28}, $\# \mathrm{Col}_3((a \tilde{\Delta}^{\pm 2})^{\wedge})=\# \mathrm{Col}_3(\hat{a})$; thus (\ref{Q2}) follows from Proposition \ref{a-prop6-5}. 
The quandle cocycle invariant of $\mathcal{S}_n(a,b)$ is computed by seeing the weights of triple points which appear when we transform the braid presentation $ab$ to $ba$ \cite{N, N5}. For the case (\ref{Q1}), the related torus-covering $T^2$-links have diagrams with the same set of weights of triple points; so their quandle cocycle invariants coincide. 
\end{proof}

\begin{theorem}\label{thm6-9}
Any $\mathcal{S}_3(a,b)$ is qdl-equivalent to one of the following: 
\begin{enumerate}
\item[$(1)$]
$ \mathcal{S}_3(c^{\pm 1}, e)$, 
\item[$(2)$]
$\mathcal{S}_3(c^{\pm 1}, \tilde{\Delta})$, 
\end{enumerate}
where $c$ is one of the following 3-braids: 
\[
e,\ \sigma_1, \ 
\sigma_1\sigma_2, \ 
\sigma_1 \sigma_2^{-1}, \ (\sigma_1 \sigma_2^{-1})^2. 
\]
\end{theorem}

\begin{proof}
By Proposition \ref{prop6-5}, we see that $F \coloneqq \mathcal{S}_3(a,b)$ is qdl-equivalent to 
$\mathcal{S}_3(c_0^{l_1} \tilde{\Delta}^{m_1}, c_0^{l_2} \tilde{\Delta}^{m_2})$ for some 3-braid $c_0$ and some integers $l_1, l_2, m_1, m_2$. 
By the relations (\ref{eq1121-2}) and (\ref{Q1}), $F$ is qdl-equivalent to $\mathcal{S}_3(c_1 \tilde{\Delta}^l, \tilde{\Delta}^m)$, where $c_1$ is a 3-braid and $l, m \in \mathbb{Z}$. 
By (\ref{Q3}), $F$ is qdl-equivalent to $\mathcal{S}_3(c_2, e)$ or $\mathcal{S}_3(c_2, \tilde{\Delta})$ for some 3-braid $c_2$. By (\ref{Q4}), $c_2$ can be replaced by $(\sigma_1^{\epsilon_1}) \sigma_2^{\epsilon_2} \sigma_1^{\epsilon_3} \sigma_2^{\epsilon_4}\cdots$, where $\epsilon_i \in \{+1, -1\}$ for each $i$. Since $d^{-1}\tilde{\Delta} \,d = \tilde{\Delta}$ for any 3-braid $d$, using (\ref{eq1121-1}) if necessary, we can assume that when $c_2$ consists of at most three letters, it is either $e$, $\sigma_1$, $\sigma_2$, $\sigma_1\sigma_2$, $\sigma_1 \sigma_2^{-1}$ or their inverses. 
Further, since $\sigma_1^{-1}\sigma_2\sigma_1=\sigma_2 \sigma_1 \sigma_2^{-1}$, by (\ref{eq1121-1}) we can identify the case $c_2=\sigma_2$ with $c_2=\sigma_1$. 

From now on we consider the case when $c_2$ consists of more than three letters. Using (\ref{eq1121-1}) if necessary, we can assume that $c_2=\sigma_1^{\epsilon_1} \sigma_2^{\epsilon_2} \sigma_1^{\epsilon_3} \sigma_2^{\epsilon_4}\cdots \sigma_2^{\epsilon_{2k}}$, where $k>1$  and $\epsilon_i \in \{+1, -1\}$ $(i=1,\ldots, k)$. 
For 3-braids $d_1$ and $d_2$, we denote $d_1 \sim_{qdl} d_2$ if $[d_1]=[d_2]$. 
If $c_2$ contains a sub-sequence $\sigma_i \sigma_j$ or $\sigma_i^{-1} \sigma_j^{-1}$ $(\{i,j\}=\{1,2\})$, then, by the braid relation $\sigma_i\sigma_j = \sigma_j \sigma_i \sigma_j \sigma_i^{-1}$ and  (\ref{Q4}), $c_2$ is replaced by a word with smaller number of letters. For example, 
\begin{eqnarray*}
\sigma_1 \sigma_2 \sigma_1^{-1} \sigma_2^{\pm1} &=& (\sigma_2 \sigma_1 \sigma_2 \sigma_1^{-1}) \sigma_1^{-1} \sigma_2^{\pm1} \\
&\ \sim_{qdl} \ & \sigma_2 \sigma_1 \sigma_2 \sigma_1 \sigma_2^{\pm1} =\sigma_2 \sigma_2 \sigma_1\sigma_2 \sigma_2^{\pm1} \\
&\sim_{qdl}& \sigma_2^{-1} \sigma_1 \sigma_2 \sigma_2^{\pm1}  \ \sim_{qdl} \ \sigma_2^{-1} \sigma_1 \sigma_2^{-1} \ \text{or}\ \sigma_2^{-1} \sigma_1. 
\end{eqnarray*}
Therefore 
we can assume that 
$\epsilon_1, \epsilon_2,\ldots, \epsilon_{2k}$ have alternating signs. 
We see that 
\begin{eqnarray*}
(\sigma_1 \sigma_2^{-1})^3 &=& \sigma_1 (\sigma_2^{-1} \sigma_1) (\sigma_2^{-1} \sigma_1) \sigma_2^{-1}\\
&=& \sigma_1 (\sigma_1 \sigma_2 \sigma_1^{-1} \sigma_2^{-1})
(\sigma_1 \sigma_2 \sigma_1^{-1} \sigma_2^{-1}) \sigma_2^{-1}
\\
&=& \sigma_1^2 \sigma_2 (\sigma_1^{-1} \sigma_2^{-1}
\sigma_1 \sigma_2) \sigma_1^{-1} \sigma_2^{-2}
\\
&=& \sigma_1^2 \sigma_2 (\sigma_2 \sigma_1^{-1})\sigma_1^{-1} \sigma_2^{-2} 
\\
&=& \sigma_1^2 \sigma_2^2 \sigma_1^{-2} \sigma_2^{-2}\\ 
 &\sim_{qdl}&   \sigma_1^{-1} \sigma_2^{-1} \sigma_1 \sigma_2 \\
 &=&  \sigma_2 \sigma_1^{-1}\\
 &=&(\sigma_1 \sigma_2^{-1})^{-1}. 
 \end{eqnarray*}
So, when $c_2$ consists of letters with alternate signs, $c_2$ is either $\sigma_1\sigma_2^{-1}$ or $(\sigma_1\sigma_2^{-1})^2$ and their inverses. 
\end{proof}

\begin{proof}[Proof of Theorem \ref{thm1-4}]
We use Theorem \ref{thm6-9}. 
For $F=\mathcal{S}_3(a,b)$ of type (2) of Theorem \ref{thm6-9}, 
 Proposition \ref{a-prop6-5} implies $\Phi_3(F)=3$. 

Let $F$ be of type (1) in Theorem \ref{thm6-9}. Note that for $F=\mathcal{S}_3(c^{\pm 1}, e)$, $\Phi_3(F)$ is the number of tri-colorings of the closure of  $c^{\pm 1}$. 
The cases of $\Phi_3(F)=27$ and $\Phi_3(F)=9$ follow from Proposition \ref{newprop5-1}. 
The other cases follow from Propositions \ref{prop5-2-28} and  \ref{prop5-2}. 
The number of tri-colorings can also be obtained by a direct computation of the rank of $A_c-I \bmod{3}$. 
\end{proof}

\section{Other results}\label{sec8}

\begin{theorem}\label{thm8-2}
Let $F$ be a surface-link. If  the quandle cocycle invariant $\Phi_3(F)(v) \in \mathbb{Z}[v,v^{-1}]/(v^3-1)$ does not have an integer value, then $F$ cannot be presented in the form of a torus-covering $T^2$-link of degree equal to or less than three. 
\end{theorem}

\begin{proof}
If a torus-covering $T^2$-link $F$ is of degree less than three, then $F$ has a diagram with no triple points; thus the quandle cocycle invariant $\Phi_3(F)(v)$ is an integer. 
Hence Theorem \ref{thm1-4} implies the required result. 
\end{proof}

We define the {\it torus-covering index} of a torus-covering $T^2$-link $F$, denoted by $\mathrm{tc} (F)$, as the smallest $n$ such that $F$ can be presented as a torus-covering $T^2$-link of degree $n$. 
We remark that for a torus-covering $T^2$-link $F$ with $N$ components, $\mathrm{tc}(F)\geq N$. 

\begin{corollary}
Let $k$ and $m$ be arbitrary integers such that  $k, m \not \equiv 0 \pmod{3}$. Then, the torus-covering $T^2$-link $\mathcal{S}_4(\sigma_1^2 \sigma_2^{3k} \sigma_3^2, \tilde{\Delta}^{m})$ has the torus-covering index $4$. 
\end{corollary}

\begin{proof}
Put $F_k=\mathcal{S}_4(\sigma_1^2 \sigma_2^{3k} \sigma_3^2, \tilde{\Delta}^{m})$. 
In \cite[Theorem 5.5]{N}, we computed $\Phi_3(F_1)$ as 
\[
\Phi_3(\mathcal{S}_4(\sigma_1^2 \sigma_2^{3} \sigma_3^2, \tilde{\Delta}^{m}))(v)=3\sum_{i=0}^2 v^{2mi^2}=3+6v^{2m}
\]
in $\mathbb{Z}[v, v^{-1}]/(v^3-1)$. By a similar calculation, we have 
\[
\Phi_3(\mathcal{S}_4(\sigma_1^2 \sigma_2^{3k} \sigma_3^2, \tilde{\Delta}^{m}))(v)=3+6v^{2km}.
\]
Thus, if $k,m \not\equiv 0 \pmod{3}$, then $\Phi_3(F_k) \not\in \mathbb{Z}$; hence Theorem \ref{thm8-2} implies that under the assumption $k,m \not\equiv 0 \pmod{3}$, the torus-covering index $\mathrm{tc}(F_k)=4$. 
\end{proof}

\section*{Acknowledgements}
H.B. was supported by JSPS KAKENHI Grant Number JP25K23338 and Research Fellowship Promoting International Collaboration, The Mathematical Society of Japan. 
I.N. was partially supported by JST FOREST Program, Grant Number JPMJFR202U.


\begin{thebibliography}{0}

\bibitem{AS}
S. Asami, S. Satoh, 
An infinite family of non-invertible surfaces in 4space, 
{\it Bull. Lond. Math. Soc.} {\bf 37}, No. 2 (2005) 285-296.

\bibitem{CJKLS}
J.S. Carter, D. Jelsovsky, S. Kamada, L. Langford, M. Saito, 
Quandle cohomology and state-sum invariants of knotted curves and surfaces, 
{\it Trans. Amer. Math. Soc.} {\bf 355} (2003) 3947--3989.  

\bibitem{CKS}
J.S. Carter, S. Kamada, M. Saito, 
{\it Surfaces in 4-space}, Encyclopaedia of Mathematical Sciences 142, Low-Dimensional Topology III, Berlin, Springer-Verlag, 2004. 

\bibitem{CF}
R.H. Crowell, R.H. Fox, {\it Introduction to Knot Theory}, Ginn and Co., Boston, 1963.

\bibitem{EN}
M. Elhamdadi, S. Nelson, {\it Quandles --- an introduction to the algebra of knots}, Student Mathematical Library, 74. American Mathematical Society, Providence, RI, 2015. 

\bibitem{Fox}
R. H. Fox, A quick trip through knot theory, in {\it Topology of 3-Manifolds} ed. M. K. Fort, Jr. (Prentice-Hall, Englewood Cliffs, NJ, 1962), pp.~120--167.

\bibitem{Lang}
S. Lang, 
{\it Algebraic number theory}, 
Graduate Texts in Mathematics, 110. New York etc.: Springer-Verlag, XIII, 1986.

\bibitem{Joyce}
D. Joyce, A classical invariants of knots, the knot quandle, {\it J. Pure Appl. Algebra} {\bf 23} (1982) 137--160.



\bibitem{Kawauchi}
A. Kawauchi, {\it A Survey of Knot Theory}, 
Birkh\"{a}user Verlag, Basel, 1996. 


\bibitem{KL}
G.~S. Kopp and J.~C. Lagarias, Ray class groups and ray class fields for orders of number fields, {\it Essent. Number Theory} {\bf 4} (2025), no.~1, 1--65. 


\bibitem{Mochizuki}
T. Mochizuki, Some calculations of cohomology groups of finite Alexander quandles, {\it J. Pure Appl. Algebra} {\bf 179} (2003) 287--330.


\bibitem{N}
I. Nakamura,  
Surface links which are coverings over the standard torus, 
 {\it Algebr. Geom. Topol.} {\bf 11} (2011) 1497--1540. 
 
 \bibitem{N2}
I. Nakamura, 
Triple linking numbers and triple point numbers of certain $T^2$-links, 
{\it Topology Appl.} {\bf 159} (2012) 1439--1447.


\bibitem{N5}
I. Nakamura, 
Unknotting numbers and triple point cancelling numbers of torus-covering knots, {\it J. Knot Theory Ramifications} {\bf 22} (2013) 1350010. 


 
\end{thebibliography}
\end{document}